\documentclass[a4paper,11pt]{amsart}

\usepackage{graphicx}
\usepackage{amsmath}
\usepackage{amssymb}
\usepackage[active]{srcltx}
\usepackage{hyperref}
\usepackage{enumerate}
\usepackage{bbm}
\usepackage{amsthm}

\newtheorem{thm}{Theorem}
\newtheorem{lem}[thm]{Lemma}
\newtheorem{prop}[thm]{Proposition}
\newtheorem{cor}[thm]{Corollary}
\newtheorem*{cor*}{Corollary}
\newtheorem*{lem*}{Lemma}

\theoremstyle{definition}

\theoremstyle{remark}
\newtheorem{remark}{Remark} 

\newtheorem*{prfdisc*}{Discussion of Proof} 

\theoremstyle{plain}

\def\CC{{\mathbb C}}

\def\HH{{\mathbb H}}
\def\NN{{\mathbb N}}

\def\RR{{\mathbb R}}

\def\ZZ{{\mathbb Z}}

\def\Stab{{\mathrm{Stab}}}

\def\scrB{{\mathcal B}}
\def\scrC{{\mathcal C}}

\def\scrF{{\mathcal F}}
\def\scrH{{\mathcal H}}
\def\scrG{{\mathcal G}}
\def\scrI{{\mathcal I}}

\def\scrM{{\mathcal M}}
\def\scrN{{\mathcal N}}

\def\scrO{{\mathcal O}}
\def\scrP{{\mathcal P}}

\def\scrS{{\mathcal S}}

\def\scrY{{\mathcal Y}}

\def\Re{\operatorname{Re}}
\def\Im{\operatorname{Im}}

\def\dim{\operatorname{dim}}
\def\dist{\operatorname{dist}}

\def\G{\operatorname{G{}}}

\def\SL{\operatorname{SL}}

\def\SO{\operatorname{SO}}

\def\PSL{\operatorname{PSL}}

\def\supp{\operatorname{supp}}

\def\Onder#1#2#3#4#5{#1 \setbox0=\hbox{$#1$}\setbox1=\hbox{$#2$}
       \dimen0=.5\wd0 \dimen1=\dimen0 \dimen2=\dp0 \dimen3=\dimen2
       \advance\dimen0 by .5\wd1 \advance\dimen0 by -#4
       \advance\dimen1 by -.5\wd1 \advance\dimen1 by -#4
       \advance\dimen2 by -#3 \advance\dimen2 by \ht1
       \advance\dimen2 by 0.3ex \advance\dimen3 by #5
        \kern-\dimen0\raisebox{-\dimen2}[0ex][\dimen3]{\box1}
       \kern\dimen1}

\newcommand{\GaG}{\Gamma\backslash G}
\newcommand{\GaH}{\Gamma\backslash \mathbb{H}}

\newcommand{\sfrac}[2]{{\textstyle \frac {#1}{#2}}}

\newcommand{\matr}[4]{\left( \begin{matrix} #1 & #2 \\ #3 & #4 \end{matrix} \right) }

\newcommand{\smatr}[4]{\left( \begin{smallmatrix} #1 & #2 \\ #3 & #4 \end{smallmatrix} \right) }

\newcommand{\fg}{\mathfrak{g}}

\newcommand{\disG}[2]{\mathsf{dist}_G(#1,#2) }
\newcommand{\disGa}[2]{\mathsf{dist}_{\Gamma\backslash G}(#1,#2) }
\newcommand{\sYGa}{\mathcal{Y}_{\Gamma}}

\newcommand{\kthe}{k_{\theta}}

\newcommand{\delGa}{\delta_{\Gamma}}

\newcommand{\mPSGgN}{\mu_{\Gamma g N}^{\mathrm{PS}}}
\begin{document}
\title[Effective equidistribution of horocycles]{Effective equidistribution of the horocycle flow on geometrically finite hyperbolic surfaces}

\author{Samuel C. Edwards}

\address{Department of Mathematics, Yale University, New Haven 06511 CT, USA}
\email{samuel.edwards@yale.edu}

 \begin{abstract}
We prove effective equidistribution of non-closed horocycles in the unit tangent bundle of infinite-volume geometrically finite hyperbolic surfaces.  
 \end{abstract}



\date{\today}



\maketitle
\section{Introduction}\label{intro}
\subsection{Background}\label{background} Let $\scrM$ be a \emph{geometrically finite hyperbolic surface}. $\scrM$ may thus be realized as a quotient $\GaH$, where $\HH=\lbrace z\in\CC\,:\,\Im(z)>0\rbrace$ is the hyperbolic upper half-space equipped with the standard Riemannian metric $ds^2=\frac{dx^2+dy^2}{y^2}$ on which $G=\PSL(2,\RR)$ acts by orientation-preserving isometries in the form of M\"obius transformations, and $\Gamma <G$ is a finitely generated torsion-free discrete subgroup of $G$. The unit tangent bundle $T^1(\scrM)$ of $\scrM$ may be identified with the homogeneous space $\GaG$. The group $G$ acts naturally on $\GaG$ by right translation, that is
\begin{equation*}
g\cdot \Gamma h:= \Gamma h g\qquad \forall g\in G,\, \Gamma h \in \GaG.
\end{equation*}
The goal of this article is to provide quantitative information about ergodic averages of orbits (with respect to the above group action) of the horospherical subgroup $N=\lbrace n_x=\pm\smatr 1 x 0 1 \,:\, x\in\RR\rbrace<G$ on $\GaG=T^1(\scrM)$ in the case that $\scrM$ has \emph{infinite volume}. 

A hyperbolic surface $\scrM=\GaH$ as above is said to have finite volume if any (and hence \emph{every}) \emph{fundamental domain} $\scrF_{\Gamma}\subset \HH$ for $\Gamma$ satisfies $\mu_{\HH}(\scrF_{\Gamma})<\infty$, where  $\mu_{\HH}$ is the $G$-invariant Borel measure on $\HH$ given by $d\mu_{\HH}(x+i y)=\frac{dx\,dy}{y^2}$. If every fundamental domain $\scrF_{\Gamma}$ has infinite $\mu_{\HH}$-measure, then $\scrM$ is said to be of infinite volume. If $\GaH$ has finite volume then $\Gamma$ is said to be a lattice.

The classification of $N$-invariant ergodic Radon measures on $\GaG$ when $\Gamma$ is a lattice (and subsequent generalizations to orbits of unipotent subgroups in finite-volume quotients of semisimple Lie groups) has a long history going back to  Furstenberg \cite{Furst}, Veech \cite{Veech}, Dani \cite{Dani,Dani2} (amongst others), and culminating in the famous results of Ratner \cite{Ratner}.

Dani and Smillie \cite{Dani3, DaniSmillie} were the first to prove \emph{equidistribution} of $N$-orbits for general lattices in $G$: they proved that
\begin{equation}\label{ds}
\lim_{T\rightarrow\infty} \frac{1}{T}\int_0^T f(\Gamma g n_t)\,dt= m_{\Gamma g N}(f)\qquad \forall \Gamma g\in \GaG,\; f\in C_c(\GaG),
\end{equation}
where either $m_{\Gamma gN}$ is  the unique $G$-invariant Borel probability measure on $\GaG$, or $\Gamma g $ a periodic point for the $N$-action and $m_{\Gamma g N}$ is the Lebesgue measure on $\Gamma gN$ normalized so as to be a probability measure. In more recent years, there has been interest in quantifying the convergence in \eqref{ds}, i.e. bounding $\left|\frac{1}{T}\int_0^T f(\Gamma g n_t)\,dt- \mu_{\Gamma g N}(f)\right|$ by some explicit function depending on $\Gamma g$, $T$ (and $f$) that decays as $T\rightarrow\infty$. Burger \cite{Burger90} proved effective equidistribution of horocycles in compact quotients $\GaG$. More generally, one may use \emph{Margulis' thickening trick} and \emph{exponential mixing of the geodesic flow}, cf.\ e.g.\ Kleinbock and Margulis \cite[Proposition 2.4.8]{KleinMarg} to prove a similar result  for the action of horospherical subgroup acting on a compact quotient  of a semisimple Lie group. For non-compact $\GaG$ (still with $G=\PSL(2,\RR)$ and $\Gamma$ a lattice), effective equidistribution of horocycles was proved by Flaminio and Forni \cite{Flam}, and Str\"ombergsson \cite{Andjmd}. See also Sarnak and Ubis \cite{SarnakUbis} for an alternative proof for $\SL(2,\ZZ)\backslash \SL(2,\RR)$. Even more recently, McAdam \cite{McAdam} proved effective equidistribution of horospherical orbits on quotients $\SL(d,\ZZ)\backslash\SL(d,\RR)$ for $d\geq 3$.  

For infinite-volume $\GaG$ the situation is more complicated.
In \cite{Burger90}, Burger suggested that any $N$-invariant ergodic Radon measure is either a multiple of the natural Lebesgue measure on a closed $N$-orbit, or is a multiple of an explicit $N$-invariant Radon measure on $\GaG$ constructed using the \emph{Patterson measure} on the \emph{limit set} of $\Gamma$. Furthermore, for \emph{convex-cocompact} $\Gamma$ with \emph{critical exponent} greater than one half, it is proved in \cite{Burger90} that this is indeed the case. Roblin \cite{Roblin} subsequently generalized Burger's construction to general CAT(-1) spaces, and associated it with mixing properties of the geodesic flow on these spaces. Further work by Winter \cite{Winter} (building on Roblin's results) confirmed that in the general setting of a rank one simple linear Lie group, every invariant ergodic Radon measure for a horospherical subgroup $H$ is indeed either the natural projection of the Haar measure on $H$ onto a closed $H$-orbit, or is given by the constructions of Burger and Roblin (or scalar multiples of these). This construction is now called the \emph{Burger-Roblin measure}.

Returning to the case of geometrically finite hyperbolic surfaces, we now recount what is known regarding the equidistribution of horocycles in $\GaG$. In addition to the classification of $N$-invariant Radon measures on $\GaG$ for $\Gamma$ convex cocompact with critical exponent greater than one half, Burger also proved \cite[Corollary of Theorem 1]{Burger90}: an equidistribution result in the form of a ratios ergodic theorem of Ces\`{a}ro averages of two-sided integrals along $N$-orbits for \emph{all} points in $\GaG$ whose $N$-orbit is not closed. Schapira \cite{Schap,SchapIMRN,Schaphalf}
 generalized and strengthened Burger's results: she proved ratio ergodic theorems for one and two-sided averages along all non-closed horocycles in the unit tangent bundle of a geometrically finite surface of pinched negative curvature. We recall that in infinite-volume ergodic theory, ratios ergodic theorems are perhaps the most natural to consider. This is related to the fact that one cannot normalize the integrals $\int_0^T f(\Gamma g n_t)\,dt$ or $\int_{-T}^Tf(\Gamma g n_t)\,dt$ uniformly over almost all $\Gamma g \in \GaG$ (with respect to the natural $G$-invariant measure on $\GaG$ induced by the Haar measure on $G$) so that the integrals converge towards $\mu_{\Gamma g N}(f)$, cf.\ \cite{Aaron1,Aaron2}.
 
Nevertheless, if one allows the normalizing factor to depend on the starting point, one \emph{can} obtain ``classical" equidistribution statements for all starting points on non-closed horocycles. It turns out that the correct normalizing factor is the so-called \emph{Patterson-Sullivan measure} on the horocycle orbit. Maucourant and Schapira \cite{MauSchap} proved this type of equidistribution for two-sided averages along all non-closed horocycles. In \cite{MoOh}, Mohammadi and Oh proved a generalization of this to non-closed horospheres in geometrically finite quotients of $\SO_0(n,1)$ for all $n\geq 2$.  Our main results, Theorems \ref{mainthm} and \ref{mainthm2}, strengthen the equidistribution result of \cite{MauSchap} further; we make it \emph{effective}, that is to say: we give a quantitative bound on the difference between a normalized integral of a test function along a non-closed horocycle and the Burger-Roblin measure of the function that decays as one lets the piece of the horocycle grow in a symmetric manner. 
\subsection{The limit set and critical exponent} Before stating our results, we first recall some important aspects of dynamics on infinite volume geometrically finite hyperbolic surfaces. We refer the reader to \cite{Schaphoro} for a more thorough exposition and further references for this material.

We start by recalling the definitions of the \emph{limit set} and \emph{critical exponent} of $\Gamma$. Let $\partial_{\infty}\HH$ denote the geometric boundary of $\HH$; i.e.\ $\partial_{\infty}\HH=\RR\cup\lbrace\infty\rbrace$. The action of $G$ has a unique continuous extension to $\partial_{\infty}\HH$ given by
\begin{equation*}
g\cdot z= \sfrac{a z + b}{cz+d}\qquad \forall g=\pm\smatr a b c d \in G,\; z\in \HH\cup\partial_{\infty}\HH.
\end{equation*}
The \emph{limit set} of $\Gamma$ is denoted $\Lambda(\Gamma)$. This is the closed, $\Gamma$-invariant, subset of $\partial_{\infty}\HH$ defined by
\begin{equation*}
\Lambda(\Gamma):=\lbrace u\in\partial_{\infty}\HH\,:\, \exists\,\lbrace\gamma_n\rbrace_{n=1}^{\infty} \subset \Gamma \;\mathrm{such\;that\;}\lim_{n\rightarrow\infty}\gamma_n\cdot i= u\rbrace.
\end{equation*}
The metric on $\HH$ induced from $ds$ is denoted $\dist$; hence, given $z,w\in\HH$, $\cosh(\dist(z,w))=1+\frac{|z-w|^2}{2\Im(z)\Im(w)}$. Using this, we define the \emph{critical exponent} $\delGa$ of $\Gamma$ by
\begin{equation*}
\delGa:=\inf \bigg\lbrace s\in\RR\,:\, \sum_{\gamma\in\Gamma} e^{-s\dist(\gamma\cdot i,i)}<\infty\bigg\rbrace.
\end{equation*}
If $\Lambda(\Gamma)$ consists of  more than two points then it must in fact be an infinite set. 
We distinguish between these two cases by saying that $\Gamma$ is \emph{elementary} if $\Lambda(\Gamma)$ consists of at most two points, and \emph{non-elementary} otherwise. It will be the non-elementary groups $\Gamma$ that will be of most interest to us.

Beardon \cite{Beardon}, Patterson \cite{Patterson2}, and Sullivan \cite{Sul1} all studied connections between $\Lambda(\Gamma)$ and $\delGa$ for infinite volume $\scrM=\GaH$, and their higher-dimensional generalisations. An important consequence of their work is that in this case, $\delGa=\dim_{\mathrm{Haus}}(\Lambda(\Gamma))\in(0,1)$.

The points of the limit set may be classified further: $\Lambda(\Gamma)$ consists of the \emph{parabolic fixed points} and \emph{radial limit points} of $\Gamma$. A point $u\in\partial_{\infty}\HH$ is a parabolic fixed point (abbreviated \emph{pfp}) if $\Stab_{\Gamma}(u)$ is conjugate to $\pm\smatr{1}{\ZZ}{0}{1}$, and is a radial limit point if there exists a \emph{geodesic ray} $\scrG\subset \HH$ \emph{tending to} $u$, a sequence $\lbrace \gamma_j\rbrace_{j=1}^{\infty}$, and $r>0$ such that $\lim_{j\rightarrow\infty}\gamma_j\cdot i =u$ and $\dist(\scrG,\gamma_j\cdot i)<r$ for all $j=1,2,\ldots$. The set of pfps of $\Gamma$ is denoted by $\Lambda_{\mathrm{pfp}}(\Gamma)$ and the set of radial limit points is denoted $\Lambda_{\mathrm{rad}}(\Gamma)$. As such,
\begin{equation*}
\Lambda(\Gamma)=\Lambda_{\mathrm{pfp}}(\Gamma)\cup\Lambda_{\mathrm{rad}}(\Gamma),
\end{equation*}
where the union is disjoint. We observe that $\Lambda_{\mathrm{pfp}}(\Gamma)$ and $\Lambda_{\mathrm{rad}}(\Gamma)$ are both $\Gamma$-invariant.

Another subgroup of $G$ that will be of importance to us is 
\begin{equation*}
A=\left\lbrace a_y=\pm\smatr{\sqrt{y}}{0}{0}{\frac{1}{\sqrt{y}}}\,:\, y\in\RR_{>0}\right\rbrace.
\end{equation*}
This subgroup is closely related to geodesics in $\HH$: given $u_1\neq u_2\in \partial_{\infty}\HH$, the geodesic  from $u_1$ tending to $u_2$ is given by $\lbrace ga_y\cdot i\,:\, y\in\RR_{>0}\rbrace$, where $g\in G$ is such that $g\cdot 0=u_1$ and $g\cdot \infty=u_2$.

We also use $A$ to define the \emph{forward and backwards visual points} of $g\in G$, $[g]^+$ and $[g]^-$, as follows:
\begin{equation*}
[g]^+:=\lim_{y\rightarrow\infty} ga_y\cdot i\in\partial_{\infty}\HH,\qquad [g]^-:=\lim_{y\rightarrow0} ga_y\cdot i\in\partial_{\infty}\HH.
\end{equation*}
Let $G_{\mathrm{rad}}$ and $G_{\mathrm{pfp}}$ denote the subsets of $G$ defined by
\begin{equation*}
G_{\mathrm{rad}}=\lbrace g\in G\,:\, [g]^+\in \Lambda_{\mathrm{rad}}(\Gamma)\rbrace,\qquad G_{\mathrm{pfp}}=\lbrace g\in G\,:\, [g]^+\in \Lambda_{\mathrm{pfp}}(\Gamma)\rbrace.
\end{equation*}
Since $\Lambda_{\mathrm{rad}}(\Gamma)$  and $\Lambda_{\mathrm{pfp}}(\Gamma)$ are both $\Gamma$-invariant, we may define subsets of $\GaG$ by
\begin{align*}
&\GaG_{\mathrm{rad}}:=\lbrace \Gamma g\in \GaG\,:\, g\in G_{\mathrm{rad}}\rbrace,\quad \GaG_{\mathrm{pfp}}:=\lbrace \Gamma g\in \GaG\,:\, g\in G_{\mathrm{pfp}}\rbrace,
\\&\qquad\qquad\qquad\qquad\quad \GaG_{\mathrm{wand}}:=\GaG\setminus\big(\GaG_{\mathrm{rad}}\cup\GaG_{\mathrm{pfp}}\big).
\end{align*}
The identity
\begin{equation*}
a_y n_t a_y^{-1}=n_{yt}\qquad \forall y\in \RR_{>0}, t\in \RR,
\end{equation*}
is of fundamental importance in the study of the dynamics of the $N$-action on $\GaG$. In particular, observe that 
\begin{equation*}
[gn_t]^+=[g]^+\qquad \forall g\in G,\;t\in \RR,
\end{equation*}
so the sets $\GaG_{\mathrm{rad}}$, $\GaG_{\mathrm{pfp}}$, and $\GaG_{\mathrm{wand}}$ are all $N$-invariant. These sets in fact characterize the $N$-orbits as follows:
\begin{enumerate}
\item $\Gamma g \in \GaG_{\mathrm{pfp}}$ $\Leftrightarrow$ $\Gamma g$ is $N$-periodic, i.e., there exists $t_0>0$ such that $\Gamma g n_{t_0}=\Gamma g $.

\item $\Gamma g \in \GaG_{\mathrm{wand}}$ $\Leftrightarrow$ $\Gamma g$ is \emph{not} $N$-periodic and $\overline{\Gamma g N}=\Gamma g N $.

\item $\Gamma g \in \GaG_{\mathrm{rad}}$ $\Leftrightarrow$ $\overline{\Gamma g N}= \GaG_{\mathrm{rad}} \cup \GaG_{\mathrm{pfp}}$.
\end{enumerate}
It is case (3) that we will be concerned with: the Burger-Roblin measure is supported on $\GaG_{\mathrm{rad}} \cup \GaG_{\mathrm{pfp}}$, and (as stated above) we intend to show the stronger statement that in this case, the $N$-orbits become equidistributed in $\GaG_{\mathrm{rad}} \cup \GaG_{\mathrm{pfp}}$ in a quantifiable manner.

We conclude this section by recalling the definition of the \emph{convex core} of $\GaG$ and \emph{convex cocompact} $\Gamma$. Let $\mathrm{hull}(\Gamma)\subset \HH$ denote the convex hull of $\Lambda(\Gamma)$, that is: $\mathrm{hull}(\Gamma)$ is the smallest (hyperbolic) convex subset of $\HH$ containing all geodesics with both endpoints in $\Lambda(\Gamma)$. Since $\Lambda(\Gamma)$ is $\Gamma$-invariant, $\mathrm{hull}(\Gamma)$ is as well. This allows us to define a subset $\mathrm{core}(\scrM)\subset \scrM=\GaH$ by
\begin{equation*}
\mathrm{core}(\scrM):=\Gamma\backslash \mathrm{hull}(\Gamma).
\end{equation*}
Observe that if $[g]^+$ and $[g]^-$ are both in $\Lambda(\Gamma)$, then $g\cdot i\in \mathrm{hull}(\Gamma)$. Since $\scrM$ is geometrically finite, $\mathrm{core}(\scrM)$ may be written as the (disjoint) union of a  compact set and at most a finite number of cuspidal regions. If $\mathrm{core}(\scrM)$ has no cusps, then $\Gamma$ is said to be \emph{convex cocompact}. Observe that $\Gamma$ is convex cocompact if and only if $\Lambda_{\mathrm{pfp}}(\Gamma)=\emptyset$, which is equivalent to $\Gamma$ having no parabolic elements. Finally, we recall the following result of Beardon \cite{Beardon}: if $\Gamma$ is \emph{not} convex cocompact, then $\delGa>\frac{1}{2}$.
\subsection{Main results}\label{resultsec}  In this section, we state the main results of this paper: Theorems \ref{mainthm} and \ref{mainthm2}. In order to do this we first introduce some more notation.

Firstly, we let $\sYGa$ denote the \emph{invariant height function} on $\GaG$. The stringent definition of $\sYGa$ will be given in Section \ref{Ydefsec}; for now we simply state some of its properties. Our interest in $\sYGa$ comes from the fact that for $\Gamma g \in \GaG$, $\sYGa(\Gamma g)$  measures ``how far" into a cusp of $\GaG$ the point $\Gamma g$ lies. This is made more precise as follows: $\sYGa$ is continuous and $\RR_{\geq 1}$-valued. For convex cocompact $\Gamma$, we have $\sYGa(\Gamma g) =1$ for all $\Gamma g\in\GaG$. For non-convex cocompact $\Gamma$, we use the hyperbolic metric $\dist$ on $\HH$ to define a metric $\mathsf{dist}_{\GaG}$ on $\GaG$ by
\begin{equation*}
\mathsf{dist}_{\GaG}(\Gamma g, \Gamma h):= \inf_{\gamma\in\Gamma}\dist(\gamma g\cdot i,h\cdot i)\qquad \forall g,\,h\in G.
\end{equation*}
We then have (cf.\ Proposition \ref{Yprops}): if $\sYGa(\Gamma g)> 1$, then $\Gamma g$ belongs to a \emph{cuspidal neighbourhood} in $\GaG$, and there exist constants $0<c_0<c_1$ such that
\begin{equation*}
c_0 e^{\mathsf{dist}_{\GaG}(\Gamma g, \Gamma e)} \leq \sYGa(\Gamma g) \leq c_1 e^{\mathsf{dist}_{\GaG}(\Gamma g, \Gamma e)}
\end{equation*}  
for all $g\in \G$ such that $\sYGa(\Gamma g)>1$.

The invariant height function will be used to quantify the speed at which the $A$-action moves elements of $\GaG$ into the cusps. This quantity will in turn govern the rate of equidistribution of the horocycles. In connection with this, we need to introduce a norm that controls the growth of functions in the cusps of $\GaG$. For $\alpha\geq 0$, define $\|\cdot\|_{N^{\alpha}}$ by
\begin{equation*}
\|f\|_{N^{\alpha}}:= \sup_{x\in\GaG} \frac{|f(x)|}{\sYGa(x)^{\alpha}} \qquad \forall f\in C(\GaG),\,\alpha\geq 0.
\end{equation*}
We let $\scrB^{\alpha}$ denote the subspace of $C(\GaG)$ consisting of functions with finite $N^{\alpha}$-norm. Observe that If $\alpha_1 \leq \alpha_2$, then $\|f\|_{N^{\alpha_1}}\geq \|f\|_{N^{\alpha_2}}$. In addition to the norm $\|\cdot\|_{N^{\alpha}}$, we will also require Sobolev norms of functions on $\GaG$. Letting $K=\mathrm{PSO}(2)$, we recall that we have the Iwasawa decompositions  $G=NAK$ and $G=KAN$. The decomposition $G=NAK$ may be used to decompose the Haar measure $\mu_G$ on $G$ as $d\mu_G(n_x a_y k)=\frac{dx\, dy\,d\mu_K(k) }{y^2}$, where $\mu_K$ is the Haar probability measure on $K$. We denote the natural projection of $\mu_G$ on $\GaG$ by $\mu_{\GaG}$. Since $\scrM$ has infinite volume, $\mu_{\GaG}$ is an infinite measure. In Section \ref{sobsec} we define $L^2(\GaG)=L^2(\GaG,\mu_{\GaG})$-Sobolev norms $\|\cdot\|_{\scrS^m((\GaG)}$ on functions on $\GaG$. The space of all functions $f$ on $\GaG$ such that $\|f\|_{\scrS^m(\GaG)}<\infty$ is denoted $\scrS^m(\GaG)$-this space essentially consists of all functions in $L^2(\GaG)$ with all Lie derivatives up to (and including) order $m$ also in $L^2(\GaG)$. 

Another quantity that affects the rate of convergence is the \emph{spectral gap}. We briefly recall some aspects of the spectral theory of the Laplace-Beltrami operator $\Delta=y^{-2}(\partial_y^2+\partial_x^2)$ on $L^2(\scrM)$ (the measure on $\scrM$ being the natural projection of $\mu_{\HH}$ to $\scrM$), due to Patterson \cite{Patterson1,Patterson2}. Firstly, the spectrum of $-\Delta$ in the interval $[0,\frac{1}{4})$ consists of finitely many (discrete) eigenvalues, and denoting these by $\lambda_i$, $i=0,\ldots,I$, we have
\begin{equation*}
0<\delGa(1-\delGa)=\lambda_0<\lambda_1 \leq \ldots\leq \lambda_I.
\end{equation*}
We define $s_1\in [\frac{1}{2},\delGa)$ by
\begin{equation*}
s_1:=\begin{cases} \sfrac{1}{2}\qquad & \mathrm{if\;}I=0\\ \sfrac{1}{2}+\sqrt{\sfrac{1}{4}-\lambda_1}\qquad&\mathrm{otherwise}.\end{cases}
\end{equation*}
Observe that $\frac{1}{2}\leq s_1<\delGa$. This will be important in Theorem \ref{mainthm}.

Finally, we introduce notation for both three measures that appear in our equidistribution statements. Given a radial point $\Gamma g \in \GaG_{\mathrm{rad}}$, the Patterson-Sullivan measure on $\Gamma g N$ is denoted $\mPSGgN$; we give the precise definition of this in Section \ref{PS-N}. Since $\Gamma g \in \GaG_{\mathrm{rad}}$, the map from $\RR$ to $\GaG$ given by $t\mapsto \Gamma g n_t$ is injective, allowing us to also view $\mPSGgN$ as a measure on $\RR$. This will be done throughout the article (often without comment). We let $B_T:=\lbrace t\in \RR\,:\,|t|\leq T\rbrace$. Using the notation just introduced, we have
\begin{equation*}
\mPSGgN(B_T)=\mPSGgN\big(\lbrace \Gamma g n_t\,:\, |t|\leq T\rbrace \big).
\end{equation*}
The Burger-Roblin measure on $\GaG$ is denoted $m_{\Gamma}^{\mathrm{BR}}$. Again, we postpone the precise definition of $m_{\Gamma}^{\mathrm{BR}}$ until later, cf.\ Section \ref{BRmsec}. For now, we recall from Section \ref{background} that $m_{\Gamma}^{\mathrm{BR}}$ is the unique (up to scaling) $N$-invariant Radon measure on $\GaG$ that is not supported on a closed horocycle. The last measure we need is the \emph{Bowen-Margulis-Sullivan,} or $\mathrm{BMS}$-measure on $\GaG$. This measure is denoted $m_{\Gamma}^{\mathrm{BMS}}$ We will actually not be required to carry out any calculations using the BMS-measure; it occurs solely as a normalizing factor in the main term of our equidistribution statements. The main fact we note about the BMS-measure on $\GaG$ is that it is finite: $m_{\Gamma}^{\mathrm{BMS}}(\GaG)<\infty$.

We can now state our main theorem:
\begin{thm}\label{mainthm}
Assume $\Gamma<G$ is geometrically finite and $\frac{1}{2}<\delGa<1$. Let $\Omega\subset \GaG$ be compact and $\alpha\in [0,\frac{1}{2})$. Then for all $\Gamma g\in \Omega\cap\GaG_{\mathrm{rad}}$, $f\in \scrS^4(\GaG)\cap\scrB^{\alpha}$, and $T\gg_{\Omega} 1$,
\begin{align*}
\frac{1}{\mPSGgN(B_T)}\!\!\int_{-T}^T \!f(\Gamma g n_t)\,dt\!= &\sfrac{m_{\Gamma}^{\mathrm{BR}}(f)}{m_{\Gamma}^{\mathrm{BMS}}(\GaG)}\!+\!O_{\Gamma,\Omega,\alpha}\!\Bigg(\!\|f\|_{\scrS^4(\GaG)}\bigg\lbrace\!\!\!\left(\sfrac{\sYGa(\Gamma g a_T)}{T}\right)^{\delGa-\frac{1}{2}}\!\log^3\!\left(\!2+\sfrac{T}{\sYGa(\Gamma g a_T)}\right)\!\!\!
\\&\qquad\qquad\qquad\qquad+\left(\sfrac{\sYGa(\Gamma g a_T)}{T}\right)^{\delGa-s_1}\!\!\!\bigg\rbrace+ \|f\|_{N_{\alpha}}\left(\sfrac{\sYGa(\Gamma g a_T)}{T}\right)^{\delGa-\frac{1}{2}} \!\!\Bigg).
\end{align*}
\end{thm}
We make some remarks:
\begin{remark}
The reason that this is an effective equidistribution statement for all radial starting points is that $\lim_{T \rightarrow\infty} \frac{\sYGa( \Gamma g a_T)}{T} =0$ for all $\Gamma g \in\GaG_{\mathrm{rad}}$. This is due to  the fact that if $\Gamma g\in \GaG_{\mathrm{rad}}$ then the geodesic segment $\Gamma g a_y$ ($y\geq 1$) returns infinitely often to some compact subset of $\GaG$ (combined with Proposition \ref{Yprops} \textit{(2)}). Theorem \ref{mainthm} thus shows that the speed of equidistribution of $\Gamma g B_T$ is governed by the cuspidal excursion rate of $\Gamma g a_T$; this is completely analogous to the situation for non-compact finite-volume quotients $\GaG$, cf.\ \cite[Theorem 1]{Andjmd}. We recall that excursion rates for geodesics are well-studied and related to approximation problems for $\Gamma$-orbits. For finite-volume $\GaG$, one has Sullivan's logarithm law \cite{Sullog} and Meli\'an and Pestana's computation of the Hausdorff dimension of the set of directions in $T^1(\scrM)$ around a given point of $\scrM$ with cuspidal excursion rate greater than a given number \cite{Melian}. In the case that $\GaG$ has infinite volume, there exist corresponding results due to Stratmann and Velani \cite{StratVel} and Hill and Velani \cite{HillVel}. 
\end{remark}
\begin{remark}
The measure $m_{\Gamma}^{\mathrm{BR}}$ is a priori only defined on $C_c(\GaG)$. However, (as will be seen in the proof of Theorem \ref{mainthm}) it does have a (unique) extension as a distribution on $\GaG$ to a linear functional on $\scrS^1(\GaG)$ (cf.\ \cite[Theorem 7.3]{LeeOh}).
\end{remark}
\begin{remark}
An interesting feature of Theorem \ref{mainthm} is that it holds for quite general functions on $\GaG$. Most previous equidistribution results for infinite-volume $\GaG$ require the test functions to be bounded or have compact support.
\end{remark}

\begin{remark}
The dependencies on the compact set $\Omega\subset \GaG$ come solely from a lower bound on $\mPSGgN(B_T)$, cf.\ Proposition \ref{muPSbd} and Corollary \ref{BTCOMP}. 
\end{remark}

A key part of the proof of Theorem \ref{mainthm} consists of calculating integrals of the \emph{base eigenfunction} along pieces of horocycles. The base eigenfunction is in $L^2(\GaG)$ if and only if $\delGa>\frac{1}{2}$. This is the reason for the requirement $\delGa>\frac{1}{2}$ in Theorem \ref{mainthm}. We recall that for $\delGa\leq\frac{1}{2}$, $\Gamma$ is convex-cocompact. This allows us to use \emph{exponential mixing} (we refer the reader to the beginning of Section \ref{deltalessthan1/2} for a more thorough discussion of these matters) and Margulis' thickening trick to also prove effective equidistribution of horocycles without the assumption $\delGa>\frac{1}{2}$. Before stating our result in this direction we introduce some more spaces of functions. For a compact subset $\Omega\subset\GaG$, let $\scrS^m(\Omega)$ denote the closure of \begin{equation*}
\lbrace f\in C_c^{\infty}(\GaG)\,:\, \supp f\subset \Omega,\,\mathrm{and}\, f|_{\partial\Omega}=0\rbrace
\end{equation*}
with respect to $\|\cdot\|_{\scrS^m(\GaG)}$.

Our effective equidistribution result for $\Gamma$ with $\delGa\leq \frac{1}{2}$ reads
\begin{thm}\label{mainthm2}
Let $\Gamma$ be non-elementary and convex cocompact. There exists $\eta_{\Gamma}>0$ such that for any compact subset $\Omega \subset \GaG$ and $\Gamma g\in \Omega\cap\GaG_{\mathrm{rad}}$,
\begin{equation*}
\frac{1}{\mPSGgN(B_T)}\!\!\int_{-T}^T \!f(\Gamma g n_t)\,dt= \frac{m_{\Gamma}^{\mathrm{BR}}(f)}{m_{\Gamma}^{\mathrm{BMS}}(\GaG)}\!+\!O_{\Gamma,\Omega,\Gamma g}\left(\|f\|_{\scrS^4(\GaG)} T^{ -\eta_{\Gamma}}\right)
\end{equation*}
for all $ f\in\scrS^4(\Omega),\;T\gg_{\Omega} 1$.
\end{thm}
\begin{remark}
As in Theorem \ref{mainthm}, the behaviour of $\Gamma g$ under the $A$-action affects the error term in the equidistribution statement. Here, it is the dependency of the implied constant on the starting point $\Gamma g$ that is determined by properties of the $A$-orbit of $\Gamma g$. Since $\Gamma$ is convex cocompact, for every $\Gamma g \in \GaG_{\mathrm{rad}}$, the set $\lbrace \Gamma g a_y\,:\, y\geq 1\rbrace$ is contained in a compact subset of $\GaG$. It is the maximal distance of this set to some fixed basepoint that determines the implied constant's dependency on the starting point, i.e. given $r>0$, the implied constant can be made uniform over all $\Gamma g\in \Omega\cap \GaG_{\mathrm{rad}}$ such that $\sup_{y\geq 1}\disGa{\Gamma g a_y}{\Gamma e}\leq r$. In particular, the implied constant can be made uniform over the set $\lbrace \Gamma g\,:\, [g]^{\pm}\in\Lambda(\Gamma)\rbrace$.
\end{remark}
\subsection{Overview of article}
The majority of the article (Sections \ref{Ysec}-\ref{thm1proof}) is devoted to the proof of Theorem \ref{mainthm}. As mentioned above, to do this, we combine Str\"ombergsson's  effective equidistribution result \cite[Theorem 1]{Andjmd} with an effective equidistribution statement for the base eigenfunctions, Theorem \ref{phiequi}. It is Theorem \ref{phiequi} that is the main technical result of the paper. 

In Section \ref{Ysec}, we define the invariant height function $\sYGa$ and state a collection of its properties that will be used throughout the rest of the article. Section \ref{harmonic} consists of a recollection of a series of facts regarding harmonic analysis on $\GaG$, in particular, the decomposition of $L^2(\GaG)$ into irreducible unitary representations, as well as a couple of Sobolev inequalities. 

The proof of Theorem \ref{phiequi} consists of a series of calculations using the \emph{Patterson-Sullivan density}. In Section \ref{PatSulSec} we recall the definition of \emph{conformal densities} on $\partial_{\infty}\HH$ and their properties. A key result here is \emph{Sullivan's shadow lemma}, which we use to bound the Patterson-Sullivan measures of certain sets in $\partial_{\infty}\HH$.

Having set up the necessary prerequisites, in Section \ref{phieqsec} we state and prove Theorem \ref{phiequi}. Str\"ombergsson's effective equidistribution result is stated in Section \ref{orthsec}, and combined with Theorem \ref{phiequi} in Section \ref{thm1proof} to prove Theorem \ref{mainthm}.

Section \ref{deltalessthan1/2} is devoted to the proof of Theorem \ref{mainthm2}. We start by recalling results of Stoyanov \cite{Stoyanov} and Oh and Winter \cite{OhWinter} on exponential mixing of the $A$-action on $\GaG$. This is used to show effective equidistribution of \emph{expanding translates} of pieces of horocycle orbits; the result we need is due to Mohammadi and Oh \cite{MohammadiOh}. Theorem \ref{mainthm2} is then proved by combining this result with Sullivan's shadow lemma. 
\subsection*{Acknowledgements} This research was funded by a scholarship from the Knut and Alice Wallenberg foundation. I would like to thank Hee Oh for asking me about this as well as for interesting and enlightening discussions, and  Andreas Str\"ombergsson for some useful comments. 
\section{The Invariant Height Function}\label{Ysec}
\subsection{The invariant height function}\label{Ydefsec}
Here we will define the \emph{invariant height function}. Much of this section is similar to \cite[Section 2]{Ed4}, however since we deal only with the case $G=\PSL(2,\RR)$, and \cite{Ed4} studies the general case $G=\SO_0(n,1)$, there are a number of simplifications. The primary reason for this is due to the fact that all cusps of $\GaH$ have \emph{full rank}, which is not necessarily the case in higher dimensions.

We start by recalling some properties regarding the action of $G$ on $\HH$.
For $\eta\in\partial_{\infty} \HH\setminus\lbrace \infty\rbrace$, define the \emph{horoball of diameter $\sigma$ based at $\eta$}, $\scrH(\eta,\sigma)\subset\HH$, by
\begin{equation*}
\scrH(\eta,\sigma):= \lbrace z\in\HH\,:\, |z-(\eta+i\sfrac{\sigma
}{2})|<\sfrac{\sigma}{2}\rbrace.
\end{equation*}
We also define horoballs at infinity $\scrH(\infty,\sigma)$ by
\begin{equation*}
\scrH(\infty,\sigma):= \lbrace z\in\HH\,:\, \Im(z)>\sigma\rbrace.
\end{equation*}
Observe that if $g\in G$ and $\eta\in\partial_{\infty} \HH$, then for any $\sigma>0$, there exists $\sigma_g>0$ such that $g\cdot \scrH(\eta,\sigma)=\scrH(g\cdot \eta, \sigma_{g})$.

Horoballs are important for studying the behaviour of functions in the cusps of $\GaG$. We will now define a function that captures the growth properties of functions in cusps in a succinct way. We follow \cite[Section 2]{Ed4} and \cite[Section 2]{Andjmd}.
Given a \emph{parabolic fixed point} (henceforth abbreviated pfp) $\eta\in\partial_{\infty}\HH$ of $\Gamma$, we define a subset $\scrN_{\eta}^{(\Gamma)}\subset G$ by
\begin{equation*}
\scrN_{\eta}^{(\Gamma)}:=\lbrace h\in G\,:\, h\cdot \eta=\infty\;\mathrm{and}\; h\, \Stab_{\Gamma}(\eta)\, h^{-1}=\pm\smatr{1}{\ZZ}{0}{1}\rbrace.
\end{equation*} 
Note that given a pfp $\eta$ of $\Gamma$, we have $\Im(h_1\cdot z)=\Im(h_2\cdot z)$ for all $z\in\HH$ and $h_1,$ $h_2\in\scrN_{\eta}^{(\Gamma)}$ (cf.\ \cite[Lemma 2]{Ed4}). Another important property is that $\scrN_{\eta}^{(\Gamma)} g=\scrN_{g^{-1}\cdot\eta}^{(g^{-1}\Gamma g)}$ (for all pfps $\eta$ of $\Gamma$ and $g\in G$). In particular, if $\eta$ is a pfp for $\Gamma$, then for all $\gamma\in\Gamma$,  $\gamma\cdot\eta$ is also a pfp for $\Gamma$, and $\scrN_{\gamma\cdot\eta}^{(\Gamma)} =\scrN_{\eta}^{(\Gamma )}\gamma^{-1}$.
We now define the \emph{invariant height function}: let $\widetilde{\sYGa}:\HH\rightarrow\RR_{>0}$ be defined by
\begin{equation}\label{Ydef}
\widetilde{\sYGa(z)}:= \sup_{\overset{\eta\in\partial_{\infty} \HH}{\eta \,\mathrm{is\,a\,pfp\,of\,\Gamma}}}\, \Im( h_{\eta}\cdot z)\qquad (h_{\eta}\in \scrN_{\eta}^{(\Gamma)}),
\end{equation}
and 
\begin{equation*}
\sYGa(z):=\max\lbrace 1, \widetilde{\sYGa(z)}\rbrace.
\end{equation*}
We will see shortly that $\sYGa$ is well-defined, i.e.\ the supremum in the definition is finite for every $z\in\HH$. Since $\Gamma$ is geometrically finite, the set of pfps for $\Gamma$ decomposes into a finite number $\kappa<\infty$ of $\Gamma$-orbits, cf.\ \cite[Lemma 3.1.4]{Bow1}, \cite[Corollary 6.5]{Bow2}. Choosing a set of representatives $\eta_1,\ldots,\eta_{\kappa}$ for the $\Gamma$-orbits, we may use the equality $\scrN_{\gamma\cdot\eta}^{(\Gamma)} =\scrN_{\eta}^{(\Gamma )}\gamma^{-1}$ to express $\sYGa$ as
\begin{equation*}
\sYGa(z)=\max\left\lbrace 1, \max_{1\leq i \leq \kappa} \sup_{\gamma\in\Gamma}\; \Im( h_{\eta_i}\gamma\cdot z)\right\rbrace \qquad (h_{\eta_i}\in \scrN_{\eta_i}^{(\Gamma)}).
\end{equation*}

Observe that $\sYGa$ is left $\Gamma$-invariant; we may thus also view it as a function on $\GaH$. Furthermore, we may view it as a left $\Gamma$-invariant and right $K$-invariant function on $G$ by the formula
\begin{equation*}
\sYGa(g):=\sYGa(g\cdot i)\qquad \forall g\in G.
\end{equation*}
The $\Gamma$-invariance allows us to also view $\sYGa$ as a function on $\GaG$. Note that $\sYGa(n_x a_y k)=\sYGa(x+i y)$ for all $x\in\RR$, $y>0$, $k\in K$. We will abuse notation slightly and use $\sYGa$ to denote the function on any of $\HH$, $\GaH$, $G$, and $\GaG$.

Several important properties of $\sYGa$ are captured in the following proposition:

\begin{prop}\label{Yprops}
$ $ 
\begin{enumerate}
\item $\sYGa(\Gamma g n_x)\leq \sYGa(\Gamma g)(1+|x|)^2$ for all $g\in G$, $x\in \RR$.
\item $\sYGa(\Gamma g a_y)\leq \sYGa(\Gamma  g)\max\lbrace y,y^{-1}\rbrace$ for all $g\in G$, $y>0$.
\item $\sYGa( g \cdot z)= \scrY_{g^{-1}\Gamma g}(z)$ for all $g\in G$, $z\in\HH$.
\item The set $\lbrace z\in\HH\,:\, \sYGa(z)>1\rbrace$ is a $\Gamma$-invariant disjoint union of horoballs based at the pfps of $\Gamma$.
\item There exist constants $0<c_0<c_1$ such that 
\begin{equation*}
c_0  e^{\mathsf{dist}_{\GaG}(\Gamma g, \Gamma e)}\leq\sYGa(\Gamma g)  \leq c_1 e^{\mathsf{dist}_{\GaG}(\Gamma g, \Gamma e)}
\end{equation*}
for all $g\in\lbrace h\in G\,:\, \sYGa(\Gamma h)>1\rbrace$.
\end{enumerate}
\end{prop}
\begin{proof}These statements are all contained (either explicitly or implicitly) in \cite[Section 2]{Ed4} and \cite[Section 2]{Andjmd} (cf.\ also \cite[Lemma 5]{Ed1}). For completeness, we give exact references and supplementary arguments. For \textit{(1)} and \textit{(2)}, see \cite[(12), (13), and the subsequent paragraph, p.\ 298]{Andjmd}. Item \textit{(3)} follows from the fact that $\scrN_{\eta}^{(\Gamma)} g=\scrN_{g^{-1}\cdot\eta}^{(g^{-1}\Gamma g)}$.

To prove \textit{(4)}, we choose two pfps $\eta_1\neq\eta_2$ of $\Gamma$ and let $\scrH(\eta_i,\sigma_i)$ be defined by $\scrH(\eta_i,\sigma_i)=h_i^{-1}\scrH(\infty,1)$, where $h_i\in \scrN_{\eta_i}^{(\Gamma_i)}$, $i=1,2$. After possibly conjugating $\Gamma$, we may assume that $\eta_1=\infty$, $h_1=\pm\smatr 1 0 0 1$, and $\Gamma_{\infty}=\pm\smatr{1}{\ZZ}{0}{1}$. Writing $h_2=\pm\smatr{a}{b}{c}{d}$, if $z=x+iy\in \scrH(\infty,1)\cap\scrH(\eta_2,\sigma_2)$, then
\begin{equation*}
\Im(h_2\cdot z)=\frac{y}{(cx+d)^2 + (cy)^2}>1.
\end{equation*}
Since $z\in\scrH(\infty,1)$, $y>1$. Observe also that since $\eta_2\neq \infty$, $h_2\not\in\Stab_{\Gamma}(\infty)$, hence $c\neq 0$, and thus
\begin{equation*}
1<\Im(h_2 \cdot z)\leq \sfrac{y}{(cy)^2}\leq\sfrac{ 1}{c^2 y}.
\end{equation*}
We then have
\begin{equation*}
\scrH(\infty,1)\cap\scrH(\eta_2,\sigma_2)\subset \lbrace z:\, 1<\Im(z)<\sfrac{1}{c^2}\rbrace.
\end{equation*}
Consider now the subgroup $\Gamma'<\Gamma$ defined by
\begin{equation*}
\Gamma'=\langle \Stab_{\Gamma}(\eta_1), \Stab_{\Gamma}(\eta_2)\rangle=\langle \pm\smatr{1}{1}{0}{1}, \pm h_2^{-1}\smatr{1}{1}{0}{1} h_2\rangle.
\end{equation*}
Now, since $h_2=\smatr a b c d$, $h_2^{-1} \smatr 1 1 0 1 h_2=\pm\smatr{\ast}{\ast}{-c^2}{\ast}$.
We now apply Shimizu's lemma (cf.\  \cite[Lemma 4]{Shim}, \cite[Lemma 1.7.3]{Miy}) to the discrete group $\Gamma'$: if $c^2<1$, then $h_2^{-1} \smatr 1 1 0 1 h_2\in\Gamma_{\infty}$. Since $\eta_2\neq \infty$, $h_2^{-1} \smatr 1 1 0 1 h_2\not\in\Gamma_{\infty}$, and hence $c^2\geq 1$, giving
\begin{equation*}
\scrH(\infty,1)\cap\scrH(\eta_2,\sigma_2)\subset \lbrace z:\, 1<\Im(z)<\sfrac{1}{c^2}\rbrace=\emptyset.
\end{equation*}
This shows that $\bigcup_{\mathrm{pfps\,}\eta} h_{\eta}^{-1}\scrH(\infty,1)$ is in fact a \emph{disjoint} union of horoballs. By \textit{(3)}, this is a $\Gamma$-invariant set. Consequently, $\sYGa$ is well-defined: if $z\in\HH\setminus \bigcup_{\mathrm{pfps\,}\eta} h_{\eta}^{-1}\scrH(\infty,1)$, then from \eqref{Ydef}, $\widetilde{\sYGa}(z)\leq 1$, and if $z\in h_{\eta}^{-1}\scrH(\infty,1)$, then $\widetilde{\sYGa}(z)=\Im(h_{\eta}\cdot z)>1$. Thus: $\lbrace z\in\HH\,:\, \sYGa(z)>1\rbrace=\bigcup_{\mathrm{pfps\,}\eta} h_{\eta}^{-1}\scrH(\infty,1)$. 

To prove \textit{(5)}, we make use of the set $\eta_1,\ldots,\eta_{\kappa}$ of ($\Gamma$-inequivalent) representatives for the set of all pfps. We assume that $z=g\cdot i$ and $\sYGa(z)>1$. By \textit{(4)}, $z\in\scrH(\eta,\sigma)$ for some pfp $\eta$ and $\sYGa(z)=\Im(h_{\eta}\cdot z)$. Using the $\Gamma$-invariance of $\sYGa$ and $\mathsf{dist}_{\GaG}$, we may assume that $z\in\scrH(\eta_j,\sigma_j)$, $1\leq j \leq \kappa$. We then have
\begin{align*}
e^{\mathsf{dist}(\Gamma g , \Gamma e)}=&e^{\inf_{\gamma\in \Gamma}\dist(\gamma\cdot z, i)}=e^{\inf_{\gamma\in \Gamma}\big(\dist(h_{\eta_j}\gamma\cdot z, i)+\dist(h_{\eta_j}\cdot i, i)\big)}\\& \leq \left(e^{\inf_{\gamma\in \Gamma}\dist(h_{\eta_j}\gamma\cdot z, i) }\right) \left(\max_{1\leq l\leq \kappa} e^{\dist(h_{\eta_l}\cdot i, i)} \right)\ll \left(e^{\inf_{\gamma\in \Stab_{\Gamma}(\eta_j)}\dist(h_{\eta_j}\gamma\cdot z, i) }\right).
\end{align*}
Now, since $h_{\eta_j}\Stab_{\Gamma}(\eta_j)h_j^{-1}=\pm\smatr{1}{\ZZ}{0}{1}$, we can find $\gamma \in \Stab_{\Gamma}(\eta_j)$ such that
\begin{equation*}
h_{\eta_j}\gamma  z=h_{\eta_j}\gamma h_{\eta_j}^{-1}\cdot (h_{\eta_j}\cdot z)=x_j+i\Im(h_{\eta_j}\cdot z),
\end{equation*}
with $|x_j|\leq \frac{1}{2}$. This gives $e^{\inf_{\gamma\in \Stab_{\Gamma}(\eta_j)}\dist(h_{\eta_j}\gamma\cdot z, i) }\leq e^{\dist(x_j+i\Im(h_{\eta_j}\cdot z),i)}\ll\Im(h_{\eta_j}\cdot z),$ and so
\begin{equation*}
e^{\mathsf{dist}(\Gamma g , \Gamma e)}\ll \Im(h_{\eta_j}\cdot z)=\sYGa(z)=\sYGa(\Gamma g). 
\end{equation*}
In the opposite direction, note that if $\gamma \not \in \Stab_{\Gamma}(\eta_j)$, then $\Im(h_{\eta_j}\gamma h_{\eta_j}^{-1}\cdot i)\leq 1$ (see the proof of \textit{(4)}). This gives
\begin{align*}
\dist(\gamma\cdot z, i)=&\dist(h_{\eta_j}z,h_{\eta_j}\cdot i)\geq \left(\dist(h_{\eta_j}\gamma\cdot z, i)-\dist(h_j\cdot i, i)\right)\\&\geq \dist(h_{\eta_j}\gamma\cdot z, i)-\left(\max_{1\leq l \leq \kappa}\dist(h_l\cdot i, i) \right)\\&=\dist(h_{\eta_j}\cdot z, h_{\eta_j}\gamma h_{\eta_j}^{-1}\cdot i)-\left(\max_{1\leq l \leq \kappa}\dist(h_l\cdot i, i) \right).
\end{align*}
Since $\Im(h_{\eta_j}\cdot z)>1$ and $\Im(h_{\eta_j}\gamma h_{\eta_j}^{-1}\cdot i)\leq 1$, 
\begin{equation*}
\dist(h_{\eta_j}\cdot z, h_{\eta_j}\gamma h_{\eta_j}^{-1}\cdot i)\geq \log\left( \sfrac{\Im(h_{\eta_j}\cdot z)}{\Im(h_{\eta_j}\gamma h_{\eta_j}^{-1}\cdot i)}\right)\geq \log(\Im(h_{\eta_j}\cdot z)).
\end{equation*}
This gives
\begin{equation*}
\dist(\gamma\cdot z, i) \geq \log(\Im(h_{\eta_j}\cdot z))-\left(\max_{1\leq l \leq \kappa}\dist(h_l\cdot i, i) \right).
\end{equation*}
For $\gamma\in \Stab_{\Gamma}(\eta_j)$, $h_{\eta_j} \gamma h_{\eta_j}^{-1}\in \ZZ$, hence 
\begin{align*}
\dist(\gamma\cdot z, i)&=\dist(h_{\eta_j}\gamma\cdot z, i)-\dist(h_j\cdot i, i)\\\geq&\left(\inf_{n\in \ZZ }\dist(h_{\eta_j}\cdot z+n, i)\right)-\left(\max_{1\leq l \leq \kappa}\dist(h_l\cdot i, i) \right)\\&\geq \log\big( \Im(h_{\eta_j}\cdot z)\big)-\left(\max_{1\leq l \leq \kappa}\dist(h_l\cdot i, i) \right).
\end{align*}
In conclusion,
\begin{equation*}
e^{\mathsf{dist}(\Gamma g , \Gamma e)}=e^{\inf_{\gamma\in \Gamma}\dist(\gamma\cdot z, i)}\geq e^{\log\big( \Im(h_{\eta_j}\cdot z)\big)-\left(\max_{1\leq l \leq \kappa}\dist(h_l\cdot i, i) \right)}\gg \Im(h_{\eta_j}\cdot z)=\sYGa(z).
\end{equation*}
\end{proof}
\section{Decomposition of $L^2(\GaG)$ and Sobolev Inequalities}\label{harmonic}
\subsection{Unitary representations}
Recall the notation from Section \ref{intro}: $-\Delta$ ($\Delta$ is the Laplace-Beltrami operator acting on $L^2(\scrM)$) has finitely many eigenvalues $\lambda_0,\ldots,\lambda_I$ in $[0,\frac{1}{4})$: $0<\delGa(1-\delGa)=\lambda_0<\lambda_1\leq\ldots\leq \lambda_I<\frac{1}{4}$, and we write $\lambda_i=s_i(1-s_i)$ with $s_i\in (\frac{1}{2},1)$, $i=0,\ldots,I$ (note thus that $s_0=\delGa$). 

We now recall the decomposition of the unitary representation $(\rho,L^2(\GaG))$ into tempered and non-tempered parts; here $\rho$ denotes right translation, i.e.\ $\big(\rho(g)f\big)(\Gamma h)=f(\Gamma hg)$ for all $g\in G$, $f\in L^2(\GaG)$, and $\Gamma h\in \GaG$. Letting $H$, $X_+$, and $X_-$ denote the following elements of the Lie algebra $\fg=\mathfrak{sl}(2,\RR)$ of $G$:
\begin{equation*}
H=\matr{1/2}{0}{0}{-1/2},\quad X_+=\matr{0}{1}{0}{0},\quad X_-=\matr{0}{0}{1}{0},
\end{equation*}
the Casimir element $\scrC$ of $\fg$ may be expressed as $\scrC=H^2-H+X_+X_-$. Identifying $L^2(\scrM)$ with the subspace $L^2(\GaG)_K\subset L^2(\GaG)$ of $\rho(K)$-invariant vectors, one observes that $\scrC$ acts on $L^2(\GaG)_K$ as $\Delta$; this allows one to combine the spectral theory of $\Delta$ on $L^2(\scrM)$ with the classification of the unitary dual of $G$ to obtain the following:
\begin{prop}\label{decomp}(cf.\ \cite[Theorem 3.1]{LeeOh})
\begin{equation*}
\big(\rho,L^2(\GaG)\big)= \bigoplus_{i=0}^{I} (\rho,\scrC_i)\oplus \big(\rho,L^2(\GaG)_{temp}\big),
\end{equation*}
where each $(\rho,\scrC_i)$ is a complementary series representation which $\scrC$ acts on the smooth vectors of by $s_i(s_i-1)$, and $\big(\rho,L^2(\GaG)_{temp}\big)$ is tempered.
\end{prop}
\subsection{Sobolev inequalities}\label{sobsec}
We start by recalling the definition of the Sobolev norms that we need. Fix a basis $X_1,X_2,X_3$ of $\fg$, and for $m\in\NN$, define
\begin{equation*}
\|f\|_{\scrS^m(\GaG)}:=\sqrt{\sum_{U} \|Uf\|_{L^2(\GaG)}^2},\qquad \forall f\in C^{\infty}(\GaG)\cap L^2(\GaG),
\end{equation*} 
where the sum runs over all monomials $U$ in the $X_i$ of order not greater than $m$ (this includes the element $``1"$ of order zero). We let $\scrS^m(\GaG)\subset L^2(\GaG)$ denote the closure (with respect to $\|\cdot\|_{\scrS^m(\GaG)}$) of the elements $f$ of $L^2(\GaG)\cap C^{\infty}(\GaG)$ with $\|f\|_{\scrS^m(\GaG)}<\infty$. Also, define $\scrS^{\infty}(\GaG):=\bigcap_{m\in\NN} \scrS^m(\GaG)$. 

Using an automorphic Sobolev inequality of Bernstein and Reznikov \cite[Proposition B.2]{Bern}, we may use $\sYGa$ and Sobolev norms to express the following pointwise bound on functions in $\scrS^2(\GaG)$:
\begin{lem}\label{Sobbdd}
\begin{equation*}
|f(\Gamma g)|\ll_{\Gamma} \|f\|_{\scrS^2(\GaG)}\sYGa(\Gamma g)^{\frac{1}{2}}\qquad \forall f\in\scrS^2(\GaG),\,g\in G.
\end{equation*}
\end{lem} 
\begin{proof}
This is \cite[Proposition 6]{Ed4}. Observe that ``$\sYGa$" in \cite{Ed4} is equal to ``$\widetilde{\sYGa}$" (cf.\ \eqref{Ydef}) here.
\end{proof}
For ``smooth enough" functions in the subrepresentations $\scrC_i$, we have the following stronger pointwise bound:
\begin{lem}\label{Compbdd} Given $i\in\lbrace 0,\ldots , I\rbrace$ and $s_i$ as in Proposition \ref{decomp},
\begin{equation*}
|f(\Gamma g)|\ll_{\Gamma} \|f\|_{\scrS^3(\GaG)}\sYGa(\Gamma g)^{1-s_i}\qquad \forall f\in\scrC_i\cap\scrS^3(\GaG),\,g\in G.
\end{equation*}
\end{lem}
\begin{proof}
This is \cite[Lemma 16]{Andjmd}. Observe that the proof there essentially follows from ``constant term" calculations in the cusps of $\GaG$. For $G=\SL(2,\RR)$ and $\Gamma$ geometrically finite, the cusps have the same structure as for the cusps in the case $\Gamma$ is a lattice (that is to say: all cusps have full rank). This enables the proof given in \cite{Andjmd} to be carried over without modification.
\end{proof}
\section{Patterson-Sullivan Densities and Measures}\label{PatSulSec}
Here we recall the definitions of the Patterson-Sullivan densities on $\partial_{\infty}\HH$ and measures on $N$-orbits in $\GaG$. Since we will require these construction for conjugations $g^{-1} \Gamma g$ $(g\in \G$) as well as for $\Gamma$, we will be (perhaps overly) careful with expressing dependencies on $\Gamma$.

\subsection{Conformal densities}
We start by recalling the definition of a conformal density. Let $H$ be a subgroup of $G$. An \emph{$H$-invariant conformal density of dimension $\delta$} is a collection $\lbrace \mu_z\rbrace_{z\in\HH}$ of finite Borel measures on $\partial_{\infty}\HH$ that satisfy
\begin{equation}\label{nuprops}
\left(\frac{d\mu_w}{d\mu_z}\right)(u)=e^{-\delta \beta_u (w,z)},\qquad h_* \mu_{z}=\mu_{h\cdot  z}\qquad \forall z,w\in\HH,\,u\in\partial_{\infty}\HH,\,h\in H.
\end{equation}
We recall the (standard) notation used here: for a measure $\mu$ on $\HH\cup \partial_{\infty}\HH$ and $g\in G$, the measure $g_*\mu$ is defined via $(g_*\mu)(A)=\mu(g^{-1}\cdot A)$ for suitable $A\subset(\HH\cup \partial_{\infty}\HH)$. Also, $\beta_u(w,z)$ denotes the \emph{Busemann cocycle}, i.e., for $u\in\partial_{\infty}\HH$,
\begin{equation*}
\beta_u(w,z):= \lim_{t\rightarrow\infty} \mathrm{dist}(w,\xi_t)-\mathrm{dist}(z,\xi_t)\qquad \forall w,\,z\in\HH,
\end{equation*}
where $\xi_t$ is any geodesic ray in $\HH$ tending to $u$.

There exists a \emph{unique up to scaling} $\Gamma$-invariant conformal density of dimension $\delGa$, called the \emph{Patterson-Sullivan density} (cf.\ \cite{Patterson2, Sul}). Given $w\in\HH$, we may realize this conformal density as the collection $\lbrace \nu_{z}^{(\Gamma,w)}\rbrace_{z\in\HH}$, where each $\nu_{z}^{(\Gamma,w)}$ is defined via the weak limit
\begin{equation}\label{nudef}
\nu_{z}^{(\Gamma,w)}:=\lim_{s\rightarrow \delGa^+} \frac{1}{\sum_{\gamma\in\Gamma} e^{-s\,\mathrm{dist}(\gamma\cdot w,w)}}\sum_{\gamma\in\Gamma} e^{-s\,\mathrm{dist}( z, \gamma\cdot w)} \delta_{\gamma\cdot w}\qquad \forall z,w\in \HH.
\end{equation}
(here $\delta_{\zeta}$ denotes the unit mass at $\zeta\in\HH$). We recall that all the measures in the Patterson-Sullivan density are supported on $\Lambda(\Gamma)$ and are non-atomic; we may thus also view it as a collection of measures on $\RR=\partial_{\infty}\HH\setminus\lbrace\infty\rbrace$.

Since the Patterson-Sullivan density is unique up to scaling, there exists a function $\scrP_{\Gamma}:\HH\rightarrow\RR_{>0}$ such that
\begin{equation}\label{nuscaling}
\nu_{z}^{(\Gamma,w)}=\scrP_{\Gamma}(w)\,\nu_{z}^{(\Gamma,i)}\qquad\forall z,w\in\HH.
\end{equation}
Note that it follows from \eqref{nudef} that $\scrP_{\Gamma}(\gamma \cdot w)=\scrP_{\Gamma}(w)$ for all $\gamma\in\Gamma$, $w\in\HH$. 
\begin{lem}\label{nuzwbound}
$ $
\begin{enumerate}[i)]
\item $\nu_z^{(\Gamma,w)}(A)\leq e^{\delGa \mathrm{dist}(z,v)}\nu_v^{(\Gamma,w)}(A)\qquad \forall z,\,w\in\HH,\,A\subset \partial_{\infty}\HH\;\;\mathrm{measurable.}$
\item $e^{-\delGa \mathrm{dist}(w,\Gamma\cdot i)} \leq \scrP_{\Gamma}(w)\leq e^{\delGa \mathrm{dist}(w,\Gamma \cdot i)}\qquad \forall w\in \HH$.
\end{enumerate}
\end{lem}
\begin{proof}
Using the observation $|\beta_u(z,v)|\leq \mathrm{dist}(z,v)$ and \eqref{nuprops}, \textit{i)} is proved as follows:
\begin{equation*}
\nu_z^{(\Gamma,w)}(A)= \int_A d\nu_z^{(\Gamma,w)}(A)(u)= \int_A e^{-\delGa\beta_u(z,v)}d\nu_v^{(\Gamma,w)}(A)\leq e^{\delGa \mathrm{dist}(z,v)}\nu_v^{(\Gamma,w)}(A).
\end{equation*}
For \textit{ii)}, note that from the definition that each $\nu_w^{(\Gamma,w)}$ is a probability measure, hence (again using $|\beta_u(z,v)|\leq \mathrm{dist}(z,v)$ and \eqref{nuprops})
\begin{equation*}
1=\int_{\partial_{\infty}\HH} d\nu_w^{(\Gamma,w)}(u)=\scrP_{\Gamma}(w) \int_{\partial_{\infty}\HH} d\nu_w^{(\Gamma,i)}(u)=\scrP_{\Gamma}(w) \int_{\partial_{\infty}\HH} e^{-\delGa\beta_u(w,i)}d\nu_i^{(\Gamma,i)}(u),
\end{equation*} 
so
\begin{equation*}
e^{-\delGa \mathrm{dist}(w,i)} \scrP_{\Gamma}(w)\leq 1 \leq e^{\delGa \mathrm{dist}(w,i)} \scrP_{\Gamma}(w).
\end{equation*}
Now using the $\Gamma$-invariance of $\scrP_{\Gamma}$, we have
\begin{equation*}
 \left(\inf_{\gamma\in \Gamma}e^{\delGa \mathrm{dist}(w,\gamma i)}\right)^{-1} \scrP_{\Gamma}(w)\leq 1 \leq \left(\inf_{\gamma\in \Gamma}e^{\delGa \mathrm{dist}(w,\gamma i)}\right) \scrP_{\Gamma}(w).
\end{equation*}
\end{proof}

Using \eqref{nudef} we obtain the following transformation rule:
\begin{lem}\label{conjlem}
For a geometrically finite group $\Gamma<G$ and $g\in G$, the Patterson-Sullivan densities of $\Gamma$ and $g^{-1} \Gamma g$ satisfy
\begin{equation*}
\nu_z^{(g^{-1} \Gamma g,w)}= (g^{-1})_*\nu_{g\cdot z}^{(\Gamma,g\cdot w)} \qquad \forall z,w\in \HH.
\end{equation*}
\end{lem} 
\subsection{Patterson-Sullivan measures on $N$-orbits}\label{PS-N}
For any $g\in G$, recall that the \emph{forward} and \emph{backward} visual maps, $[g]^+$ and $[g]^-$, of $g$ are defined by
\begin{equation*}
[g]^+:=\lim_{y\rightarrow \infty} ga_y\cdot i \in \partial_{\infty}\HH\qquad [g]^-:=\lim_{y\rightarrow 0} ga_y\cdot i \in \partial_{\infty}\HH.
\end{equation*}
Let $\Gamma g\in \GaG_{\mathrm{rad}}$, that is $[g]^+\in \Lambda_{\mathrm{rad}}(\Gamma)$. The map from $N$ to $\GaG$ given by
\begin{equation*}
n\mapsto \Gamma g n\qquad\forall n\in N
\end{equation*}
is then injective . This allows us to ``lift" measures in the Patterson-Sullivan density to a measure on $\Gamma g N \subset \GaG$ by
\begin{equation*}
d\mu_{\Gamma g N}^{\mathrm{PS}}(\Gamma g n):= e^{\delGa \beta_{[gn]^-}(z,gn \cdot z)}\,d\nu_{z}^{(\Gamma,i)}([gn ]^-)\qquad \forall n \in N,
\end{equation*}
where $z\in \HH$. Since $\Gamma g N\leftrightarrow \RR$, we may view this as a measure on $\RR$ (or $N$) via
\begin{equation*}
d\mu_{\Gamma g N}^{\mathrm{PS}}(x)=d\mu_{\Gamma g N}^{\mathrm{PS}}(n_x)= e^{\delGa \beta_{[g n_x]^-}(z,gn_x\cdot z)}\,d\nu_{z}^{(\Gamma,i)}([gn_x]^-) =\, e^{\delGa \beta_{g\cdot x}(z,g\cdot (z+x))}\,d\nu_{z}^{(\Gamma,i)}(g\cdot x).
\end{equation*}
The properties in \eqref{nuprops} show that $\mu_{\Gamma g N}^{\mathrm{PS}}$ is well-defined, i.e.\ independent of the chosen representative of $\Gamma g$ and basepoint $z\in \HH$. Furthermore, by \cite[Lemma 2.4]{MoOh}, $\mu_{\Gamma g N}^{\mathrm{PS}}$ is an infinite measure (on $\RR$ alt.\ $N$). Recall that $B_T= \lbrace  t\in\RR\,:\, |t|\leq T\rbrace$.

\begin{lem}\label{conjBT}
\begin{equation*}
\mu_{\Gamma g N}^{\mathrm{PS}}(B_T)=\frac{\mu_{ (g^{-1} \Gamma g) e N}^{\mathrm{PS}}(B_T)}{\scrP_{\Gamma}(g\cdot i)}\qquad \forall g\in G_{\mathrm{rad}}, \,T\geq 0.
\end{equation*}
\end{lem}

\begin{proof}
Using the definition of $\mu_{ (g^{-1} \Gamma g) e N}^{\mathrm{PS}}$, \eqref{nuprops}, \eqref{nuscaling}, and Lemma \ref{conjlem} (as well as the fact that $\delta_{g^{-1} \Gamma g}=\delGa$):
\begin{align*}
\mu_{ (g^{-1} \Gamma g) e N}^{\mathrm{PS}}(B_T)&=\int_{-T}^T\,d \mu_{ (g^{-1} \Gamma g) e N}^{\mathrm{PS}}(x)=\int_{-T}^Te^{\delGa\beta_{x}(z, x+z)}\,d\nu_z^{(g^{-1} \Gamma g,i)}(x)
\\&=\int_{-T}^Te^{\delGa\beta_{x}(z, x+z)}\,d\big((g^{-1})_*\nu_{g\cdot z}^{(\Gamma,g\cdot i )}\big)(x)=\scrP_{\Gamma}(g\cdot i)\int_{-T}^Te^{\delGa\beta_{x}(z, x+z)}\,d\nu_{g\cdot z}^{(\Gamma,i )}(g\cdot x)
\\&=\scrP_{\Gamma}(g\cdot i)\int_{-T}^Te^{\delGa\beta_{x}(z, x+z)}e^{-\delGa\beta_{g\cdot x}(g\cdot z, z)}\,d\nu_{ z}^{(\Gamma,i )}(g\cdot x).
\end{align*}
Since $g$ acts as an isometry on $\HH$, $\beta_x(z,x+z)=\beta_{g\cdot x}(g\cdot z, g\cdot(x+z))$ for all $g\in G$, $z\in\HH$, $x\in \partial_{\infty}\HH$. This, combined with the cocycle property of $\beta$, gives
\begin{equation*}
\beta_{x}(z, x+z)-\beta_{g\cdot x}(g\cdot z, z)=\beta_{g\cdot x}(g\cdot z, g\cdot(x+z))-\beta_{g\cdot x}(g\cdot z, z)=\beta_{g\cdot x}(z,g\cdot(z+x)),
\end{equation*}
and so (once again using the definition of $\mu_{ \Gamma  g N}^{\mathrm{PS}}$)
\begin{equation*}
\mu_{ (g^{-1} \Gamma g) e N}^{\mathrm{PS}}(B_T)=\scrP_{\Gamma}(g\cdot i)\int_{-T}^Te^{\delGa \beta_{g\cdot x}(z,g\cdot(z+x))}\,d\nu_{ z}^{(\Gamma,i )}(g\cdot x)=\scrP_{\Gamma}(g\cdot i)\mu_{ \Gamma  g N}^{\mathrm{PS}}(B_T).
\end{equation*}
\end{proof}
\begin{remark}
Observe that since $\scrP_{\Gamma }(\gamma g \cdot i)=\scrP_{\Gamma}(g\cdot i)$, both sides of the equation in Lemma \ref{conjBT} are therefore independent of the representative chosen from $\Gamma g$. We will henceforth also view $\scrP_{\Gamma}$ as a function on $\GaG$ by defining $\scrP_{\Gamma}(\Gamma g):=\scrP_{\Gamma}(g\cdot i)$. Note that Lemma \ref{nuzwbound} \textit{ii)} then gives
\begin{equation*}
 e^{-\delGa\mathsf{dist}_{\GaG}(\Gamma g, \Gamma e)}\leq \scrP_{\Gamma}(\Gamma g)\leq  e^{\delGa\mathsf{dist}_{\GaG}(\Gamma g, \Gamma e)}.
\end{equation*}
\end{remark}
Lemma \ref{conjBT} will be used together with the following observation: if $\infty\in\Lambda_{\mathrm{rad}}(\Gamma)$, then $\Gamma e$ is radial, and
\begin{equation}\label{Ninftmeas}
\mu_{\Gamma e N}^{\mathrm{PS}}(B_T)=\int_{-T}^T e^{\delGa \beta_{ x}(i,x+i)}\,d\nu_{i}^{(\Gamma,i)}( x)=\int_{-T}^T (1+x^2)^{\delGa}\,d\nu_{i}^{(\Gamma,i)}(x).
\end{equation}
We make one final observation regarding $\mPSGgN$, which is proved using calculations similar to those in the proof of Lemma \ref{conjBT}:
\begin{lem}\label{mPSscaling}
For all $T>0$ and $\scrI\subset \RR$ measurable,
\begin{equation*}
\mPSGgN(\scrI)=T^{\delGa}\mu^{\mathrm{PS}}_{\Gamma g a_T N}(\sfrac{\scrI}{T}),
\end{equation*}
where $\frac{\scrI}{T}=\lbrace \frac{x}{T}\,:\,x\in \scrI\rbrace$.
\end{lem}
\subsection{The Lebesgue density}\label{Lebsec}
In Sections \ref{BRmsec} and \ref{deltalessthan1/2} we will also require the \emph{Lebesgue density}. This is a $G$-invariant density of dimension one, and denoted $\lbrace m_z\rbrace_{z\in\HH}$. Each $m_z$ is non-atomic, again allowing us to view them as measures on $\RR$. Defining a measure $\widetilde{\mu}$ on $\RR$ by
\begin{equation*}
d\widetilde{\mu}(u)=(1+u^2)dm_i(u)=e^{\beta_u(i,u+i)}dm_i(u)\qquad \forall u\in\RR,
\end{equation*}
we obtain that for all $y\in\RR_{>0}$, $x,u\in\RR$,
\begin{align*}
d\widetilde{\mu}(yu+x)=&d\widetilde{\mu}(n_xa_y\cdot u)=e^{\beta_{n_xa_y\cdot u}(i,yu+x+i)}dm_i(n_xa_y\cdot u)
\\&=e^{\beta_{ u}(a_y^{-1}n_x^{-1}\cdot i,u+\frac{i}{y})}dm_{(n_x a_y)^{-1}\cdot i}( u)=e^{\beta_{ u}(a_y^{-1}n_x^{-1}\cdot i,u+\frac{i}{y})}e^{-\beta_{u}(a_y^{-1}n_x^{-1}\cdot i,i)}dm_{ i}( u)
\\&=e^{\beta_{u}( i,u+\frac{i}{y})}dm_{ i}( u)= y(1+u^2)\,dm_i(u)= y\,d\widetilde{\mu}(u).
\end{align*}
The measure $\widetilde{\mu}$ must therefore be a scalar multiple of the Lebesgue measure. This allows us to therefore assume that the density $\lbrace m_z\rbrace_{z\in\HH}$ has been scaled so that $dm_i(u)=\frac{du}{1+u^2}$.
\subsection{The shadow lemma}\label{Shadowsec}
We will use a version of \emph{Sullivan's Shadow Lemma} to obtain (both upper and lower) bounds for the $\nu_z^{(\Gamma,i)}$-measures of certain subsets of $\partial_{\infty}\HH$. We start by recalling the definition of the \emph{base eigenfunction} $\phi_0\in L^2(\GaG)$, cf. \cite{Patterson2,Sul}. This is a $\rho(K)$-invariant function in $\scrC_{s_0}\cap \scrS^{\infty}(\GaG)$ (cf. Proposition \ref{decomp}), and is given by the formula
\begin{equation}\label{phi0def}
\phi_0(\Gamma g)=\scrN_{\Gamma}\int_{\partial_{\infty}\HH} e^{-\delGa \beta_u(g\cdot i,i)}\,d\nu_i^{(\Gamma,i)}(u),
\end{equation}
where the constant $\scrN_{\Gamma}\in \RR_{>0}$ is chosen so that $\|\phi_0\|_{L^2(\GaG)}=1$. Observe that $\phi_0(\Gamma g)>0$ for all $g\in G$. Since $\phi_0\in\scrS^3(\GaG)\cap \scrC_{s_0}$, by Lemma \ref{Compbdd}, 
\begin{equation}\label{phi0ptbdd}
|\phi_0(\Gamma g)|\ll_{\Gamma} \|\phi_0\|_{\scrS^3(\GaG)} \sYGa(\Gamma g)^{1-\delGa}\ll_{\Gamma} \sYGa(\Gamma g)^{1-\delGa}
\end{equation}
(recall that $s_0=\delGa$).

For $w\in\HH$ and $r>0$, let $B_r(w)$ denote the open (hyperbolic) ball of radius $r$ around $w$. Given another point $z\in \HH$, we let $\scrO_z(w,r)\subset \partial_{\infty}\HH$ denote the \emph{shadow of $B_r(w)$ seen from $z$}; this is the set of points $u\in \partial_{\infty}\HH$ with the property that the geodesic segment from $z$ to $u$ intersects $B_r(w)$. Observe that since $G$ acts by isometry on $\HH$, $g\cdot\scrO_z(w,r)=\scrO_{g\cdot z}(g\cdot w,r)$.

We have the following result, due to Sullivan, cf.\ \cite[Section 7]{Sul1}: 
\begin{lem}\label{shadlem}
For all $z,\,w\in\HH$, $r>0$,
\begin{equation*}
\nu_z^{(\Gamma,i)}\big(\scrO_z(w,r)\big)\ll_{\Gamma} e^{2\delGa r-\delGa\mathrm{dist}(z,w)} \sYGa(w)^{1-\delGa}. 
\end{equation*}
\end{lem}
\begin{proof}
Using \eqref{phi0def}, \eqref{nuprops}, and writing $w=x+i y$, we have
\begin{align*}
\phi_0(\Gamma n_x a_y)=&\scrN_{\Gamma}\int_{\RR} e^{-\delGa \beta_u(x+i y, i)}\, d\nu_i^{(\Gamma,i)}(u)=\scrN_{\Gamma}\int_{\RR} e^{-\delGa \beta_u(w, z)}\, d\nu_z^{(\Gamma,i)}(u)\\&\geq \scrN_{\Gamma} \int_{\scrO_z(w,r)} e^{\delGa \beta_u( z,w)}\, d\nu_z^{(\Gamma,i)}(u).
\end{align*}
Now, for all $u\in \scrO_{z}(w,r)$,
\begin{equation*}
\mathrm{dist}(z,w)-2 r \leq \beta_u(z,w)\leq \mathrm{\dist}(z,w),
\end{equation*}
hence 
\begin{equation*}
 \phi_0(\Gamma n_x a_y)\geq \scrN_{\Gamma} \int_{\scrO_z(w,r)} e^{\delGa \beta_u(x+i y, z)}\, d\nu_z^{(\Gamma,i)}(u)\geq \scrN_{\Gamma} e^{\delGa( \mathrm{dist}(z,w)-2r)}\nu_z^{(\Gamma,i)}\big(\scrO_z(w,r)\big).
\end{equation*}
By \eqref{phi0ptbdd}, we then have
\begin{equation*}
\nu_z^{(\Gamma,i)}\big(\scrO_z(w,r)\big)\leq \scrN_{\Gamma}^{-1} e^{-\delGa( \mathrm{dist}(z,w)-2r)}\phi_0(\Gamma n_x a_y)\ll_{\Gamma} e^{-\delGa( \mathrm{dist}(z,w)-2r)}\sYGa(n_xa_y)^{1-\delGa}.
\end{equation*}
\end{proof}
The following is a more or less straightforward consequence of Lemmas \ref{conjlem} and \ref{shadlem}:
\begin{lem}\label{Tinftymeasure}
\begin{equation*}
\nu_i^{(g^{-1}\Gamma g,i)}\big( \lbrace x\in\RR\,:\, |x|\geq T\rbrace\big)\ll_{\Gamma}\scrP_{\Gamma}(\Gamma g) T^{-\delGa} \sYGa(\Gamma g a_T)^{1-\delGa}\qquad \forall g\in G,\,T\geq 1.
\end{equation*}
\end{lem}
\begin{proof}
Observe that $\scrO_i( i T,\mathrm{arcsinh}(1))= \lbrace x\in\RR\,:\, |x|\geq T\rbrace \cup\lbrace \infty\rbrace$. By Lemmas \ref{conjlem} and \ref{shadlem}, we have \begin{align*}
\nu_i^{(g^{-1}\Gamma g,i)}\big( \lbrace x\in\RR&\,:\, |x|\geq T\rbrace\big)= \nu_i^{(g^{-1}\Gamma g,i)}\big(\scrO_i( i T,\mathrm{arcsinh}(1))\big)
\\&=(g^{-1}_*\nu_{g\cdot i}^{(\Gamma,g\cdot i)}\big(\scrO_i( i T,\mathrm{arcsinh}(1))\big)=\scrP_{\Gamma}(\Gamma g)\nu_{g\cdot i}^{(\Gamma,i)}\big(g\cdot\scrO_i(i T,\mathrm{arcsinh}(1))\big)
\\&=\scrP_{\Gamma}(\Gamma g)\nu_{g\cdot i}^{(\Gamma,i)}\big(\scrO_{g\cdot i}(g\cdot i T,\mathrm{arcsinh}(1))\big)
\\&\ll_{\Gamma} \scrP_{\Gamma}(\Gamma g)e^{-\delGa\mathrm{dist}(g\cdot i, g\cdot iT)} \sYGa(g\cdot iT)^{1-\delGa}
\\&=\scrP_{\Gamma}(\Gamma g)e^{-\delGa\mathrm{dist}(i,  iT)} \sYGa(\Gamma g a_{T})^{1-\delGa}.
\end{align*}
The proof is completed by noting that $e^{-\delGa\mathrm{dist}( i,  iT)}= T^{-\delGa}$.
\end{proof}

The following proposition gives a bound on the $\nu$-measures of certain subsets of $\RR$:
\begin{lem}\label{diffmeas}
\begin{align*}
&\nu_i^{(g^{-1}\Gamma g,i)}\big(\lbrace u\in\RR\,:\,(1-\epsilon)T\leq |u|\leq (1+\epsilon)T\rbrace \big)\ll_{\Gamma} \scrP_{\Gamma}(\Gamma g)\epsilon^{2\delGa-1} T^{-\delGa}\sYGa(\Gamma g a_{T})^{1-\delGa}
\end{align*}
for all $g\in G$, $T\geq 2$, $0\leq\epsilon\leq\frac{1}{2}$. 
\end{lem}
\begin{proof}
We prove the bound for the interval $[(1-\epsilon)T,(1+\epsilon)T]$; the negative interval is dealt with in a completely symmetric manner.  Given $r>0$ such that 
\begin{equation}\label{shadcontain}
[(1-\epsilon)T,(1+\epsilon)T]\subset \scrO_i(T+i \epsilon T,r),
\end{equation}
by Lemmas \ref{conjlem} and \ref{shadlem}, we then have
\begin{align*}
\nu_i^{(g^{-1}\Gamma g,i)}\big([(1-\epsilon)T,&(1+\epsilon)T]\big)\leq \nu_i^{(g^{-1}\Gamma g,i)}\big(\scrO_i(T+i \epsilon T,r)\big)=(g^{-1})_*\nu_{g\cdot i}^{(\Gamma,g\cdot i)}\big(\scrO_i(T+i \epsilon T,r)\big)
\\&=\scrP(\Gamma g)\nu_{g\cdot i}^{(\Gamma,i)}\big(g\cdot \scrO_i(T+i \epsilon T,r)\big)=\scrP(\Gamma g)\nu_{g\cdot i}^{(\Gamma,i)}\big( \scrO_{g\cdot i}(g\cdot(T+i \epsilon T),r)\big)
\\&\ll_{\Gamma}\scrP_{\Gamma}(\Gamma g)e^{2\delGa r-\delGa\mathrm{dist}(g\cdot i,g\cdot(T+i \epsilon T))} \sYGa\big(g\cdot( T+i \epsilon T)\big)^{1-\delGa}
\\&=\scrP_{\Gamma}(\Gamma g)e^{2\delGa r-\delGa\mathrm{dist}( i,T+i \epsilon T)} \sYGa\big(\Gamma gn_T a_{\epsilon T}\big)^{1-\delGa}
\\&=\scrP_{\Gamma}(\Gamma g)e^{2\delGa r-\delGa\mathrm{dist}( i,T+i \epsilon T)} \sYGa\big(\Gamma ga_Tn_1 a_{\epsilon}\big)^{1-\delGa}.
\end{align*}
By \textit{(1)} and \textit{(2)} of Proposition \ref{Yprops}, $\sYGa\big(\Gamma ga_Tn_1 a_{\epsilon}\big)^{1-\delGa}\ll \epsilon^{\delGa-1}\sYGa(\Gamma ga_T)^{1-\delGa}$. Furthermore,
\begin{align*}
e^{-\delGa\mathrm{dist}( i,T+i \epsilon T)}=\left( \frac{\sqrt{T^2+(\epsilon T-1)^2}+\sqrt{T^2+(\epsilon T+1)^2}}{2\sqrt{\epsilon T}}\right)^{-2\delGa}\leq \left( \frac{T}{\sqrt{\epsilon T}}\right)^{-2\delGa}=\epsilon^{\delGa} T^{-\delGa}.
\end{align*}
We thus have
\begin{equation*}
\nu_i^{(g^{-1}\Gamma g)}\big([(1-\epsilon)T,(1+\epsilon)T]\big)\ll_{\Gamma} e^{2\delGa r}T^{-\delGa}\epsilon^{2\delGa-1}\sYGa(\Gamma ga_T)^{1-\delGa}.
\end{equation*}
In order to complete the proof, we need to find an $r>0$ satisfying \eqref{shadcontain}. Observe that $B_r(T+i\epsilon T)$ is a Euclidean ball centred at $T+i\cosh(r)\epsilon T$ with radius $\sinh(r)\epsilon T$. The points on the geodesic rays from $i$ to $(1\pm\epsilon)T$ are given by
\begin{equation*}
\scrG_{T,\epsilon}^{\pm}\left\lbrace z\in\HH\,:\, \left|z-\sfrac{T(1\pm \epsilon)-\frac{1}{T(1\pm\epsilon)}}{2}\right| =\sfrac{T(1\pm \epsilon)+\frac{1}{T(1\pm\epsilon)}}{2}\right\rbrace,
\end{equation*}
respectively. If $\scrG_{T,\epsilon}^{\pm}$ have non-empty intersections with $B_r(T+i \epsilon T)$, then $[T(1-\epsilon),T(1+\epsilon)]$ is contained in $\scrO_i(T+i\epsilon T, r)$, i.e.\ if the following two inequalities are satisfied:
\begin{align*}
&\left|T+i \epsilon T\cosh(r)-\sfrac{T(1- \epsilon)-\frac{1}{T(1-\epsilon)}}{2}\right| <\epsilon T\sinh(r)+\sfrac{T(1- \epsilon)+\frac{1}{T(1-\epsilon)}}{2},
\\
&\left|T+i \epsilon T\cosh(r)-\sfrac{T(1+ \epsilon)-\frac{1}{T(1+\epsilon)}}{2}\right| +\epsilon T\sinh(r)>\sfrac{T(1+ \epsilon)+\frac{1}{T(1+\epsilon)}}{2}.
\end{align*}
These inequalities are fulfilled if
\begin{equation*}
1\mp\frac{2\epsilon(1\pm\epsilon)}{1+T^2(1\pm\epsilon)^2}<\sinh(r),
\end{equation*}
so taking $r=\mathrm{arcsinh}(5)$ suffices for all relevant $T$ and $\epsilon$.
\end{proof}

Since we normalize the integral over $B_T$ in Theorem \ref{mainthm} by $\frac{1}{\mPSGgN(B_T)}$, we will require a \emph{lower} bound on $\mPSGgN(B_T)$.

We first introduce some more notation: for $u\in \partial_{\infty}\HH$ and $t\geq 0$, let $h_t(u)$ be the point on the geodesic segment from $i$ tending to $u$ at distance $t$ from $i$. Let $\scrS(u,t)\subset\partial_{\infty}\HH$ denote the set of points whose orthogonal projection onto the geodesic from $i$ to $u$ lie between $h_t(u)$ and $u$. Observe that since $K=\mathrm{Stab}_G(i)$, we have $k\cdot h_t(u)=h_t(k\cdot u)$ and $k\cdot \scrS( u,t)=\scrS( k\cdot u,t)$ for all $k\in K$ and $u\in\partial_{\infty}\HH$.

\begin{thm}\label{Shadow} (cf.\ \cite[Theorem 2]{StratVel}, \cite[Theorem 3.2]{Schap})
There exist $0<c_0<c_1$ such that
\begin{equation*}
c_0 e^{-\delGa t}\sYGa(h_t(\eta))^{1-\delGa}\leq\nu_i^{(\Gamma,i)}\big(\scrS(\eta,t)\big)\leq c_1 e^{-\delGa t}\sYGa(h_t(\eta))^{1-\delGa}\qquad \forall t\geq 0,\, \eta\in \Lambda.
\end{equation*}
\end{thm} 
\begin{remark}
Here we have simply used Proposition \ref{Yprops} \textit{(5)} to simply express the results from \cite{Schap,StratVel} using the invariant height function.
\end{remark}
\begin{prop}\label{muPSbd}
There exist continuous functions $C_{\Gamma},\,D_{\Gamma}\,:\, G\rightarrow \RR_{>0}$ such that
\begin{equation*}
\mu_{\Gamma g N}^{\mathrm{PS}}(B_T)\geq C_{\Gamma}(g) T^{\delGa} \sYGa(\Gamma a_T)^{1-\delGa}
\end{equation*}
for all $g\in G_{\mathrm{rad}}$ and $T\geq D_{\Gamma}(g)$.
\end{prop}
\begin{proof}
Using Lemma \ref{conjBT}, we have
\begin{align*}
\mu_{\Gamma g N}^{\mathrm{PS}}(B_T)=\frac{\mu_{ (g^{-1} \Gamma g) e N}^{\mathrm{PS}}(B_T)}{\scrP_{\Gamma}(\Gamma g)}.
\end{align*}
Now, $\infty\in\Lambda_{\mathrm{rad}}(g^{-1}\Gamma g)$, so
\begin{align*}
\mu_{ (g^{-1} \Gamma g) e N}^{\mathrm{PS}}(B_T)=\int_{-T}^{T} (1+u^2)^{\delGa}\,d\nu_{i}^{(g^{-1} \Gamma g,i)}( u).
\end{align*}
We now choose some $R\geq 2$ (depending on $\Gamma g$ and later to be specified further), and note that by \eqref{Ninftmeas}
\begin{align}\label{muPSbTbdd}
\mu_{ (g^{-1} \Gamma g) e N}^{\mathrm{PS}}&(B_T)\geq \int_{\lbrace u\,:\, \frac{T}{R}\leq |u| \leq T\rbrace} (1+u^2)^{\delGa}\,d\nu_{i}^{(g^{-1} \Gamma g,i)}( u)
\\\notag\geq& \left(1+\left(\sfrac{T}{R}\right)^2\right)^{\delGa}\int_{\lbrace u\,:\, \frac{T}{R}\leq |u| \leq T\rbrace} \,d\nu_{i}^{(g^{-1} \Gamma g,i)}( u)
\\\notag=&\left(1+\left(\sfrac{T}{R}\right)^2\right)^{\delGa}\left(  \nu_{i}^{(g^{-1} \Gamma g,i)}(\lbrace u\in\RR\,:\, |u| \geq \sfrac{T}{R}\rbrace)-\nu_{i}^{(g^{-1} \Gamma g,i)}(\lbrace u\in\RR\,:\, |u| \geq T\rbrace)\right).
\end{align}
Let $g=ka_yn_x$. Then by Lemma \ref{conjlem}, for any $S\geq 1$, we have
\begin{align}\label{Sset}
\nu_{i}^{(g^{-1} \Gamma g,i)}(\lbrace u\in\RR\,:\, |u| \geq S\rbrace)=&\nu_{g\cdot i}^{ (\Gamma,g\cdot i )}\big(g\cdot \lbrace u\in\RR\,:\, |u| \geq S\rbrace \big)
\\\notag&=\scrP_{\Gamma}(\Gamma g)\nu_{g\cdot i}^{ (\Gamma,i )}\big(k\cdot \lbrace y(x+u)\in\RR\,:\, |u| \geq S\rbrace \big).
\end{align}
Assuming 
\begin{equation}\label{Sassump}
y(x-S)\leq -1 < 1\leq y(x+S)
\end{equation}
(i.e. $|x|\leq S-\frac{1}{y}$), we let
\begin{equation*}
S_{-}:=\min\lbrace |y(x-S)|,|y(x+S)|\rbrace=y(S-|x|),\quad S_+:=\max\lbrace |y(x-S)|,|y(x+S)| \rbrace = y(S+|x|).
\end{equation*} 
We then have
\begin{equation*}
k\cdot \lbrace u\in\RR\,:\, |u| \geq S_+\rbrace\subset g\cdot \lbrace u\in\RR\,:\, |u| \geq S\rbrace\subset k\cdot \lbrace u\in\RR\,:\, |u| \geq S_-\rbrace.
\end{equation*}
Observe now that $k\cdot \infty= [g]^+\in \Lambda_{\mathrm{rad}}(\Gamma)$. Furthermore, $\lbrace u\in \RR\,:\, |u|\geq S\rbrace=\scrS_i(\infty,\log S)$ (for all $S\geq 1$), hence
\begin{equation*}
k\cdot\lbrace u\in\RR\,:\, |u| \geq S_{\pm}\rbrace=k\cdot\scrS_i(\infty,\log S_{\pm})=\scrS_{k\cdot i}(k\cdot\infty,\log S_{\pm})=\scrS_{ i}([g]^{+},\log S_{\pm}).
\end{equation*}
Returning to \eqref{Sset}, we now have
\begin{equation*}
\scrP_{\Gamma}(\Gamma g)\nu_{g\cdot i}^{ (\Gamma,i )}\big(\scrS_{ i}([g]^{+},\log S_{+})\big)\!\leq\!\nu_{i}^{(g^{-1} \Gamma g,i)}(\lbrace u\in\RR\,:\, |u| \geq S\rbrace)\!\leq\! \scrP_{\Gamma}(\Gamma g)\nu_{g\cdot i}^{ (\Gamma,i )}\big(\scrS_{ i}([g]^{+},\log S_{-}) \big),
\end{equation*}
and so Lemma \ref{nuzwbound} gives
\begin{align}\label{nuineqs}
&\scrP_{\Gamma}(\Gamma g)e^{-\delGa\mathrm{dist}(g\cdot i,i)}\nu_{ i}^{ (\Gamma,i )}\big(\scrS_{ i}([g]^{+},\log S_{+})\big)
\\\notag&\quad\leq\nu_{i}^{(g^{-1} \Gamma g,i)}(\lbrace u\in\RR\,:\, |u| \geq S\rbrace)
\\&\notag \qquad\leq  \scrP_{\Gamma}(\Gamma g)e^{\delGa\mathrm{dist}(g\cdot i,i)}\nu_{ i}^{ (\Gamma,i )}\big(\scrS_{ i}([g]^{+},\log S_{-}) \big).
\end{align}
Keeping the notation $g=k a_y n_x$, we assume that $\frac{T}{R}$ satisfies the conditions placed on the variable $S$ in \eqref{Sassump}. Note that $T$ then also fulfils these assumptions. Combining \eqref{nuineqs}, \eqref{Sset}, and \eqref{muPSbTbdd}, we have
\begin{align*}
\mu_{\Gamma g N}^{\mathrm{PS}}(B_T)\geq (1+(\sfrac{T}{R})^2)^{\delGa}\Bigg(& e^{-\delGa\mathrm{dist}(g\cdot i,i)}\nu_{ i}^{ (\Gamma,i )}\big(\scrS_{ i}([g]^{+},\log \lbrace \sfrac{T}{R}\rbrace_+)\big)
\\&\qquad\qquad- e^{\delGa\mathrm{dist}(g\cdot i,i)}\nu_{ i}^{ (\Gamma,i )}\big(\scrS_{ i}([g]^{+},\log \lbrace T\rbrace_-)\big)\Bigg),
\end{align*}
where
\begin{align*}
&\lbrace \sfrac{T}{R}\rbrace_+= y(\sfrac{T}{R}+|x|)
\\&\lbrace T\rbrace_-=y(T-|x|).
\end{align*}
Now let $0<c_0<c_1$ be the constants from Theorem \ref{Shadow}. Using both the upper and lower bounds from the same theorem, we obtain
\begin{align}\label{lowbdd1}
\mu_{\Gamma g N}^{\mathrm{PS}}(B_T)\geq (1+(\sfrac{T}{R})^2)^{\delGa}\Bigg(& c_0 e^{-\delGa\mathrm{dist}(g\cdot i,i)}\left(\lbrace \sfrac{T}{R}\rbrace_+\right)^{-\delGa}\sYGa\big(h_{\log\lbrace \sfrac{T}{R}\rbrace_+}([g]^+ )\big)^{1-\delGa}\\\notag &-c_1 e^{\delGa\mathrm{dist}(g\cdot i,i)}\left(\lbrace T\rbrace_-\right)^{-\delGa}\sYGa\big(h_{\log\lbrace T\rbrace_-}([g]^+ )\big)^{1-\delGa}\Bigg).
\end{align}
Since $[g]^+=k\cdot \infty$, $h_t([g]^+)= ka_{e^t}\cdot i$, and so
\begin{align*}
&h_{\log\lbrace \sfrac{T}{R}\rbrace_+}([g]^+ )= ka_{\lbrace \sfrac{T}{R}\rbrace_+}\cdot i=g a_T \big(a_y n_x a_T \big)^{-1}a_{\lbrace \sfrac{T}{R}\rbrace_+}\cdot i=g a_T n_{-\frac{x}{T}}a_{\lbrace \sfrac{T}{R}\rbrace_+/(yT)}\cdot i
\\&h_{\log\lbrace T \rbrace_-}([g]^+ )= ka_{\lbrace T\rbrace_-}\cdot i=g a_T \big(a_y n_x a_T \big)^{-1}a_{\lbrace T\rbrace_-}\cdot i=g a_T n_{-\frac{x}{T}}a_{\lbrace T\rbrace_-/(yT)}\cdot i.
\end{align*}
By Proposition \ref{Yprops} \textit{(1)} and \textit{(2)}, for all $Y>0$,
\begin{equation*}
\frac{\sYGa(\Gamma g a_T)}{(1+\sfrac{|x|}{T})^2\max\lbrace \sfrac{Y}{yT},\sfrac{yT}{Y}\rbrace}\leq\sYGa\big(g a_T n_{-\frac{x}{T}}a_{Y/(yT)}\cdot i\big)\leq \sYGa(\Gamma g a_T)(1+\sfrac{|x|}{T})^2\max\lbrace \sfrac{Y}{yT},\sfrac{yT}{Y}\rbrace.
\end{equation*}
In particular,
\begin{align*}
\sYGa\big(h_{\log\lbrace \sfrac{T}{R}\rbrace_+}([g]^+ )\big)\geq \frac{\sYGa(\Gamma g a_T)}{(1+\sfrac{|x|}{T})^2\max\left\lbrace \sfrac{\lbrace \sfrac{T}{R}\rbrace_+}{yT},\sfrac{yT}{\lbrace \sfrac{T}{R}\rbrace_+}\right\rbrace},
\end{align*}
and
\begin{align*}
\sYGa\big(h_{\log\lbrace T\rbrace_-}([g]^+ )\big) \leq  \sYGa(\Gamma g a_T)(1+\sfrac{|x|}{T})^2\max\left\lbrace \sfrac{\lbrace T\rbrace_-}{yT},\sfrac{yT}{\lbrace T\rbrace_-}\right\rbrace.
\end{align*}
Using these bounds in \eqref{lowbdd1}, we have
\begin{equation*}
\mu_{\Gamma g N}^{\mathrm{PS}}(B_T)\geq T^{\delGa}\sYGa( \Gamma g a_T)^{1-\delGa}\times (\ast),
\end{equation*} 
where ``$(\ast)$" equals
\begin{align}\label{Tbdd}
\notag\frac{T^{\delGa}}{R^{2\delGa}}\Bigg(&\frac{c_0 e^{-\delGa\mathrm{dist}(g\cdot i,i)}\left(\lbrace \sfrac{T}{R}\rbrace_+\right)^{-\delGa}}{(1+\sfrac{|x|}{T})^{2-2\delGa}\max\left\lbrace \sfrac{\lbrace \sfrac{T}{R}\rbrace_+}{yT},\sfrac{yT}{\lbrace \sfrac{T}{R}\rbrace_+}\right\rbrace^{1-\delGa}}
\\\notag&\quad-c_1 e^{\delGa\mathrm{dist}(g\cdot i,i)}\left(\lbrace T\rbrace_-\right)^{-\delGa}(1+\sfrac{|x|}{T})^{2-2\delGa}\max\left\lbrace \sfrac{\lbrace T\rbrace_-}{yT},\sfrac{yT}{\lbrace T\rbrace_-}\right\rbrace^{1-\delGa} \Bigg)
\\&=\frac{e^{\delGa\mathrm{dist}(g\cdot i,i)}}{R^{2\delGa}}\Bigg(\frac{c_0 e^{-2\delGa\mathrm{dist}(g\cdot i,i)}\left( \frac{T}{\lbrace\sfrac{T}{R}\rbrace_+}\right)^{\delGa}}{(1+\sfrac{|x|}{T})^{2-2\delGa}\max\left\lbrace \sfrac{\lbrace \sfrac{T}{R}\rbrace_+}{yT},\sfrac{yT}{\lbrace \sfrac{T}{R}\rbrace_+}\right\rbrace^{1-\delGa}}
\\\notag&\qquad\quad-c_1 \left(\sfrac{T}{\lbrace T\rbrace_-}\right)^{\delGa}(1+\sfrac{|x|}{T})^{2-2\delGa}\max\left\lbrace \sfrac{\lbrace T\rbrace_-}{yT},\sfrac{yT}{\lbrace T\rbrace_-}\right\rbrace^{1-\delGa} \Bigg).
\end{align}
Since $T$ and $\frac{T}{R}$ both satisfy \eqref{Sassump}, we have
\begin{equation*}
\frac{1}{y}\leq \frac{T}{R}\leq T, \qquad |x|\leq \frac{T}{R}-\frac{1}{y}\leq T-\frac{1}{y},
\end{equation*} 
and hence
\begin{align*}
&1+\frac{|x|}{T}\leq 2,
\\& \frac{T}{\lbrace\sfrac{T}{R}\rbrace_+}=\frac{T}{y(\frac{T}{R}+|x|)}\geq \frac{T}{y(\frac{2T}{R}-\frac{1}{y})}\geq \frac{R}{y},
\\&\sfrac{T}{\lbrace T\rbrace_-}=\frac{T}{y(T-|x|)}=\frac{T}{y(\frac{T}{R}-|x| + T-\frac{T}{R})}\leq \frac{T}{1+yT(1-\frac{1}{R})}\leq \frac{R}{y(R-1)}\leq \frac{2}{y},
\\&\max\left\lbrace \sfrac{\lbrace \sfrac{T}{R}\rbrace_+}{yT},\sfrac{yT}{\lbrace \sfrac{T}{R}\rbrace_+}\right\rbrace=\max \left \lbrace \sfrac{1}{R}+\sfrac{|x|}{T} ,\frac{1}{\frac{1}{R}+\frac{|x|}{T}}\right\rbrace \leq \max \left\lbrace \sfrac{2}{R}, R\right\rbrace=R,
\\&\max\left\lbrace \sfrac{\lbrace T\rbrace_-}{yT},\sfrac{yT}{\lbrace T\rbrace_-}\right\rbrace=\max\left\lbrace 1-\sfrac{|x|}{T},y\times \sfrac{T}{\lbrace T\rbrace_-}\right\rbrace \leq \max\left\lbrace 1, y\times  \frac{R}{y(R-1)} \right\rbrace =\frac{R}{R-1}\leq 2.
\end{align*}
Entering these bounds into \eqref{Tbdd} yields
\begin{align*}
\mu_{\Gamma g N}^{\mathrm{PS}}(B_T)&\geq T^{\delGa}\sYGa( \Gamma g a_T)^{1-\delGa} \frac{e^{\delGa\mathrm{dist}(g\cdot i,i)}}{R^{2\delGa}}\Bigg(\frac{c_0 e^{-2\delGa\mathrm{dist}(g\cdot i,i)}\left( \frac{R}{y}\right)^{\delGa}}{2^{2-2\delGa}R^{1-\delGa}}-c_1 \left( \frac{2}{y}\right)^{\delGa}2^{2-2\delGa}2^{1-\delGa} \Bigg)\\=&  T^{\delGa}\sYGa( \Gamma g a_T)^{1-\delGa} \frac{e^{\delGa\mathrm{dist}(g\cdot i,i)}}{y^{\delGa}R^{2\delGa}}\Bigg(2^{2\delGa-2}c_0 e^{-2\delGa\mathrm{dist}(g\cdot i,i)}R^{2\delGa-1}-2^{3-2\delGa} c_1  \Bigg).
\end{align*}
Since $2\delGa-1>0$, there exists $R_0\geq 0$ such that
\begin{equation*}
2^{2\delGa-2}c_0 e^{-2\delGa\mathrm{dist}(g\cdot i,i)}R_0^{2\delGa-1}-2^{3-2\delGa} c_1 =2^{3-2\delGa} c_1,
\end{equation*}
so choosing $R= 2+ R_0$ gives
\begin{equation*}
\mu_{\Gamma g N}^{\mathrm{PS}}(B_T)\geq  T^{\delGa}\sYGa( \Gamma g a_T)^{1-\delGa} \left( \frac{e^{\delGa\mathrm{dist}(g\cdot i,i)}2^{3-2\delGa} c_1}{y^{\delGa}(2+R_0)^{2\delGa}}\right).
\end{equation*}
Observe that $e^{\delGa\mathrm{dist}(g\cdot i , i)}\geq y^{\delGa}$, hence
\begin{align*}
\left( \frac{e^{\delGa\mathrm{dist}(g\cdot i,i)}2^{3-2\delGa} c_1}{y^{\delGa}(2+R_0)^{2\delGa}}\right)\gg_{\Gamma}&\frac{1}{(2+R_0)^{2\delGa}}\\&\gg\frac{1}{\left( 2+ \left( \frac{2^{6-4\delGa}c_1}{c_0}\right)^{\frac{1}{2\delGa-1}}e^{\frac{2\delGa}{2\delGa-1}\mathrm{dist}(g\cdot i ,i)}\right)^{2\delGa}}\gg_{\Gamma} e^{-\frac{4\delGa^{2}}{2\delGa-1}\mathrm{dist}(g\cdot i,i)},
\end{align*}
giving
\begin{equation*}
\mu_{\Gamma g N}^{\mathrm{PS}}(B_T)\gg_{\Gamma} e^{-\frac{4\delGa^{2}}{2\delGa-1}\mathrm{dist}(g\cdot i,i)}  T^{\delGa}\sYGa( \Gamma g a_T)^{1-\delGa}.
\end{equation*}
This bound is proved under the assumption $\frac{T}{R}\geq |x|+\frac{1}{y}$ (cf.\ \ref{Sassump}), i.e. 
\begin{equation*}
T\geq   (|x|+\sfrac{1}{y})\left( 2+ \left( \frac{2^{6-4\delGa}c_1}{c_0}\right)^{\frac{1}{2\delGa-1}}e^{\frac{2\delGa}{2\delGa-1}\mathrm{dist}(g\cdot i ,i)}\right).
\end{equation*}
\end{proof}
\begin{cor}\label{BTCOMP}
Let $\Omega\subset \GaG$ be compact. Then
\begin{equation*}
\mu_{\Gamma g N}^{\mathrm{PS}}(B_T)\gg_{\Omega}  T^{\delGa} \sYGa(\Gamma a_T)^{1-\delGa}
 \qquad \forall \,\Gamma g \in \Omega\cap \GaG_{\mathrm{rad}},\;T\gg_{\Omega} 1.
\end{equation*}
\end{cor}

\section{Effective Equidistribution of the Base Eigenfunctions}\label{phieqsec}
We will now prove the effective equidistribution of the \emph{base eigenfunctions} $\phi_n$ ($n\in\ZZ)$. Recall that each $\phi_n$ is a unit vector in $L^2(\GaG)$ of $K$-type $2n$. As a starting point, we will use expressions for the $\phi_n$ in terms of integrals against a measure in the Patterson-Sullivan density. The explicit formulas we need have been developed by Lee and Oh in \cite[Section 3]{LeeOh}. For $\theta\in \RR/\pi\ZZ$, let $\kthe=\smatr{\cos\theta}{\sin\theta}{-\sin\theta}{\cos\theta}$.
\begin{prop}(\cite[Theorem 3.3]{LeeOh}) \label{phinexp}
\begin{equation*}
\phi_n(\Gamma n_x a_y \kthe)=\scrN_{\Gamma}e^{2i n \theta} \sfrac{\sqrt{\Gamma(1-\delGa)\Gamma(|n|+\delGa)}}{\sqrt{\Gamma(\delGa)\Gamma(|n|+1-\delGa)}}\int_{\RR} \left( \sfrac{(u^2+1)y}{(x-u)^2+y^2}\right)^{\delGa} \left( \sfrac{x-u-iy}{x-u+iy}\right)^n\,d\nu_i^{(\Gamma,i)}(u)
\end{equation*}
for all $n\in\ZZ,\,x\in\RR,\,y>0,\,\theta\in\RR/\pi\ZZ$.
\end{prop}
\begin{remark}
The constant $\scrN_{\Gamma}$ (cf.\ \eqref{phi0def}) does not appear in the formula given in \cite{LeeOh}. This is due to the fact that we require $\nu_i^{(\Gamma,i)}$ to be a probability measure, wheras this is not the case in \cite{LeeOh}. We thus obtain that ``$\nu_j$" in \cite{LeeOh} equals our $\scrN_{\Gamma}\,\nu_i^{(\Gamma,i)}$.
\end{remark}
\begin{cor}\label{phinconj}
\begin{equation*}
\phi_n(\Gamma g n_x a_y \kthe)=\frac{\scrN_{\Gamma} e^{2i n \theta}}{\scrP_{\Gamma}(\Gamma g)} \sfrac{\sqrt{\Gamma(1-\delGa)\Gamma(|n|+\delGa)}}{\sqrt{\Gamma(\delGa)\Gamma(|n|+1-\delGa)}}\int_{\RR} \left( \sfrac{(u^2+1)y}{(x-u)^2+y^2}\right)^{\delGa} \left( \sfrac{x-u-iy}{x-u+iy}\right)^n\,d\nu_i^{(g^{-1} \Gamma g,i)}(u)
\end{equation*}
for all $g\in G$, $n\in\ZZ,\,x\in\RR,\,y>0,\,\theta\in\RR/\pi\ZZ$.
\end{cor}
\begin{proof}
For all $g,h\in G$, using \eqref{nuscaling} and Lemma \ref{conjlem}, we have
\begin{align*}
\phi_0(\Gamma gh)=&\scrN_{\Gamma}\int_{\RR} e^{-\delGa\beta_u(gh\cdot i,i)}d\nu_i^{(\Gamma,i)}(u)=\frac{\scrN_{\Gamma}}{\scrP_{\Gamma}(\Gamma g)} \int_{\RR} e^{-\delGa\beta_u(gh\cdot i , i)}e^{-\delGa\beta_u(i , g\cdot i)}\,d\nu_{g\cdot i}^{(\Gamma,g\cdot i)}(u)
\\&=\frac{\scrN_{\Gamma}}{\scrP_{\Gamma}(\Gamma g)} \int_{\RR} e^{-\delGa\beta_u(gh\cdot i , g\cdot i)}\,d\nu_{g\cdot i}^{(\Gamma,g\cdot i)}(u)
=\frac{\scrN_{\Gamma}}{\scrP_{\Gamma}(\Gamma g)} \int_{\RR} e^{-\delGa\beta_{g^{-1}\cdot u}(h\cdot i ,  i)}\,d\nu_{g\cdot i}^{(\Gamma,g\cdot i)}(u)
\\&=\frac{\scrN_{\Gamma}}{\scrP_{\Gamma}(\Gamma g)} \int_{\RR} e^{-\delGa\beta_{ u}(h\cdot i ,  i)}\,d\nu_{g\cdot i}^{(\Gamma,g\cdot i)}(g\cdot u)=\frac{\scrN_{\Gamma}}{\scrP_{\Gamma}(\Gamma g)} \int_{\RR} e^{-\delGa\beta_{ u}(h\cdot i ,  i)}\,d\nu_{i}^{(g^{-1}\Gamma g, i)}(u).
\end{align*}
For $h=n_x a_y \kthe$, $e^{-\delGa\beta_u(n_x a_y \kthe\cdot i , i)}=\left( \sfrac{(u^2+1)y}{(x-u)^2+y^2}\right)^{\delGa}$, so the formula holds for $n=0$. Following the proof of \cite[Theorem 3.3]{LeeOh}, the remaining cases follow from applying the \emph{raising and lowering operators} to the function $h\mapsto  e^{-\delGa\beta_{ u}(h\cdot i ,  i)}$ on $G$.
\end{proof}

It follows from the formulas above that $|\phi_n(\Gamma g)|\ll_{\Gamma} \phi_0(\Gamma g)$ for all $n\in \ZZ$, $g\in G$. 

Before stating the main result of this section, we make some auxiliary definitions: let
\begin{equation*}
c_n(\delGa):=\sfrac{\sqrt{\Gamma(1-\delGa)\Gamma(|n|+\delGa)}}{\sqrt{\Gamma(\delGa)\Gamma(|n|+1-\delGa)}}\qquad \forall n\in\ZZ,
\end{equation*}
\begin{equation*}
\psi_n(t):= \left(\frac{1}{t^2+1}\right)^{\delGa}\left(\frac{t-i}{t+i} \right)^n\qquad \forall t\in \RR,\, n\in \ZZ,
\end{equation*}
and
\begin{align*}
\kappa_n(\delGa):=\int_{\RR} \psi_n(t)\,dt=&\frac{4^{1-\delGa}\pi(-1)^n\Gamma(2\delGa-1)}{\Gamma(\delGa+n)\Gamma(\delGa-n)}=\frac{4^{1-\delGa}\pi\Gamma(2\delGa-1)\Gamma(|n|+1-\delGa)}{\Gamma(\delGa)\Gamma(1-\delGa)\Gamma(|n|+\delGa)}
\\ &=\kappa_0(\delGa)\frac{\Gamma(\delGa)\Gamma(|n|+1-\delGa)}{\Gamma(1-\delGa)\Gamma(|n|+\delGa)}=\frac{\kappa_0(\delGa)}{c_n(\delGa)^2} \quad \forall n\in \ZZ.
\end{align*}
Observe that $c_n(\delGa)\ll c_{0}(\delGa)$ and $|\kappa_n(\delGa)|\leq \kappa_0(\delGa)$.
Using the $c_n$s and $\kappa_n$s, we define the following functional on $\scrS^1(\GaG)$:
\begin{equation}\label{Mdef}
\scrM_{\Gamma}(f):=\sum_{n\in \ZZ} \scrN_{\Gamma}c_n(\delGa)\kappa_n(\delGa) \langle f, \phi_n\rangle_{L^2(\GaG)}=\sum_{n\in \ZZ} \frac{\scrN_{\Gamma}\kappa_0(\delGa)}{c_n(\delGa)} \langle f, \phi_n\rangle_{L^2(\GaG)}\quad \forall f\in \scrS^1(\GaG).
\end{equation}
We also have the following basic fact that will be used without comment throughout the proof of the main result of this section:
\begin{lem}\label{Rint}
\begin{equation*}
\int_{-R}^{R}\psi_n(t)\,dt= \kappa_{n}(\delGa)+O_{\delGa}(R^{1-2\delGa})\qquad \forall R>0
\end{equation*}
and
\begin{equation*}
\int_{\lbrace|t|\geq R\rbrace}\psi_n(t)\,dt=O_{\delGa}(R^{1-2\delGa})\qquad \forall R>0.
\end{equation*}
Both implied constants are independent of $n$.
\end{lem}
We now come to the main result of this section, which is essentially an effective equidistribution statement for the base eigenfunctions:
\begin{thm}\label{phiequi}
For all $\Gamma g\in \GaG_{\mathrm{rad}},\,T\geq 4,\, n\in \ZZ$,
\begin{equation}\label{phiequieq}
\int_{-T}^T \phi_n(\Gamma g n_t)\,dt =\mu_{\Gamma g N}^{\mathrm{PS}}(B_T)\scrN_{\Gamma}c_n(\delGa)\kappa_{n}(\delGa)+O_{\Gamma}\!\left( \mu_{\Gamma g N}^{\mathrm{PS}}(B_T)T^{\frac{1}{2}-\delGa}+\sYGa(\Gamma g a_T)^{1-\delGa} T^{\frac{1}{2}} \right).
\end{equation}
\end{thm}
\begin{proof}
Using Corollary \ref{phinconj}, Lemma \ref{conjBT}, and \eqref{Ninftmeas}, we have
\begin{align*}
\int_{-T}^{T}& \phi_n(\Gamma g n_t)\,dt-\scrN_{\Gamma}c_n(\delGa)\kappa_{n}(\delGa)\mPSGgN(B_T)
\\&=\frac{\scrN_{\Gamma}c_{\delGa}(n)}{\scrP_{\Gamma}(\Gamma g)}\bigg(\int_{-T}^{T} \int_{\RR}\left( \sfrac{(u^2+1)}{(t-u)^2+1}\right)^{\delGa} \left( \sfrac{t-u-i}{t-u+i}\right)^n \,d\nu_i^{(g^{-1}\Gamma g,i)}(u)\,dt
\\&\qquad\qquad\qquad\qquad\qquad\qquad-\kappa_n(\delGa)\int_{-T}^T(1+u^2)^{\delGa}\,d\nu_i^{(g^{-1}\Gamma g,i)}(u)\bigg)
\\&= \frac{\scrN_{\Gamma}c_{\delGa}(n)}{\scrP_{\Gamma}(\Gamma g)}\int_{\RR}\left( \int_{-T}^{T} \psi_n(t-u)\,dt - \kappa_{n}(\delGa)\mathbbm{1}_{[-T,T]}(u)\right)(1+u^2)^{\delGa}\,d\nu_i^{(g^{-1}\Gamma g,i)}(u)
\\&=\frac{\scrN_{\Gamma}c_{\delGa}(n)}{\scrP_{\Gamma}(\Gamma g)}\int_{\RR}\left( \int_{-u-T}^{T-u} \psi_n(t)\,dt -\kappa_{n}(\delGa)\mathbbm{1}_{[-T,T]}(u)\right)(1+u^2)^{\delGa}\,d\nu_i^{(g^{-1}\Gamma g,i)}(u)
\end{align*}
(note that $\sup_{u\in \RR} \max_{t\in[-T,T]} \frac{u^2+1}{(t-u)^2+1}\ll_T 1 $ and $\nu$ is a finite measure; this permits the  interchanging of the order of integration). We now choose some $\epsilon$, $0<\epsilon<1$, and split the integral over $\RR$ as follows
\begin{align*}
\int_{\RR}= \int_{\lbrace u\,:\, |u|\leq (1-\epsilon)T\rbrace} +\int_{\lbrace u\,:\,(1-\epsilon)T\leq|u|\leq (1+\epsilon)T\rbrace}+\int_{\lbrace u\,:\,  (1+\epsilon)T\leq |u|\leq 2T\rbrace}+\int_{u\,:\,\lbrace |u|\geq 2T\rbrace}. 
\end{align*}
We bound each of these four integrals in turn:
\begin{enumerate}[I:]
\item $\lbrace u\,:\,|u|\leq (1-\epsilon)T\rbrace$. Since $\mathbbm{1}_{[-T,T]}(u)\equiv 1$, the integral we are interested in is
\begin{equation*}
\int_{-(1-\epsilon)T}^{(1-\epsilon)T}\left( \int_{-u-T}^{T-u} \psi_n(t)\,dt -\kappa_{n}(\delGa)\right)(1+u^2)^{\delGa}\,d\nu_i^{(g^{-1}\Gamma g,i)}(u).
\end{equation*} 
Using $|u|\leq(1-\epsilon)T$, 
\begin{equation*}
\int_{-u-T}^{T-u} \psi_n(t)\,dt=\int_{-\epsilon T}^{\epsilon T}\psi_n(t)\,dt+\int_{\epsilon T}^{T-u}\psi_n(t)\,dt+\int_{u- T}^{-\epsilon T}\psi_n(t)\,dt=\kappa_n(\delGa)+O\big( (\epsilon T)^{1-2\delGa})\big),
\end{equation*}
hence 
\begin{align}\label{int1bdd}
\frac{\scrN_{\Gamma}c_{\delGa}(n)}{\scrP_{\Gamma}(\Gamma g)}&\int_{-(1-\epsilon)T}^{(1-\epsilon)T}\left( \int_{-u-T}^{T-u} \psi_n(t)\,dt -\kappa_{n}(\delGa)\right)(1+u^2)^{\delGa}\,d\nu_i^{(g^{-1}\Gamma g,i)}(u)
\\\notag&=O_{\Gamma}\left( \frac{(T\epsilon)^{1-2\delGa}}{\scrP_{\Gamma}(\Gamma g)}\int_{-(1-\epsilon)T}^{(1-\epsilon)T}(1+u^2)^{\delGa}\,d\nu_i^{(g^{-1}\Gamma g,i)}(u)\right)
\\\notag&=O_{\Gamma}\left( \frac{(T\epsilon)^{1-2\delGa}}{\scrP_{\Gamma}(\Gamma g)}\mu_{(g^{-1}\Gamma g) e N}^{\mathrm{PS}}(B_{(1-\epsilon)T})\right)
\\\notag&=O_{\Gamma}\left( (T\epsilon)^{1-2\delGa}\mPSGgN(B_{(1-\epsilon)T})\right)=O_{\Gamma}\left( (T\epsilon)^{1-2\delGa}\mPSGgN(B_{T})\right),
\end{align}
where Lemma \ref{conjBT} and \eqref{Ninftmeas} were again used. 
\item $\lbrace u\,:\,(1-\epsilon)T\leq |u|\leq (1+\epsilon)T\rbrace$. Here we use the bound
\begin{equation*}
\left| \int_{-u-T}^{T-u} \psi_n(t)\,dt -\kappa_{n}(\delGa)\mathbbm{1}_{[-T,T]}(u) \right|\leq 2\kappa_0(\delGa)=O_{\Gamma}(1).
\end{equation*}
Assuming $\epsilon\leq \frac{1}{2}$, we now use Proposition \ref{diffmeas}:
\begin{align}\label{int2bdd}
\frac{\scrN_{\Gamma}c_{\delGa}(n)}{\scrP_{\Gamma}(\Gamma g)}&\int_{\lbrace u\,:\,(1-\epsilon)T\leq|u|\leq (1+\epsilon)T\rbrace}\left( \int_{-u-T}^{T-u} \psi_n(t)\,dt -\kappa_{n}(\delGa)\right)(1+u^2)^{\delGa}\,d\nu_i^{(g^{-1}\Gamma g,i)}(u)
\\\notag&=O_{\Gamma}\left( \frac{T^{2\delGa}}{\scrP_{\Gamma}(\Gamma g)}  \nu_i^{(g^{-1}\Gamma g,i)}\big(\lbrace u\in\RR\,:\,(1-\epsilon)T\leq |u|\leq (1+\epsilon)T\rbrace \big)\right)\\\notag&= O_{\Gamma}\left( T^{\delGa}\epsilon^{2\delGa-1} \sYGa(\Gamma g a_{T})^{1-\delGa} \right).
\end{align}
\item $\lbrace u\,:\,(1+\epsilon)T\leq |u|\leq 2T\rbrace$. For $u$ in this range we have
\begin{align*}
&\left| \int_{-u-T}^{T-u} \psi_n(t)\,dt -\kappa_{n}(\delGa)\mathbbm{1}_{[-T,T]}(u) \right|=\left| \int_{-u-T}^{T-u} \psi_n(t)\,dt  \right|\\&\leq \int_{-u-T}^{T-u} \psi_0(t)\,dt=\int_{|u|-T}^{T+|u|} \psi_0(t)\,dt\leq \int_{\lbrace |t|\geq \epsilon T\rbrace} \psi_0(t)\,dt=O\big( (\epsilon T)^{1-2\delGa}\big).
\end{align*}
Lemma \ref{Tinftymeasure} gives
\begin{align}\label{int3bdd}
\frac{\scrN_{\Gamma}c_{\delGa}(n)}{\scrP_{\Gamma}(\Gamma g)}&\int_{\lbrace u\,:\,(1+\epsilon)T\leq|u|\leq 2T\rbrace}\left( \int_{-u-T}^{T-u} \psi_n(t)\,dt -\kappa_{n}(\delGa)\right)(1+u^2)^{\delGa}\,d\nu_i^{(g^{-1}\Gamma g,i)}(u)
\\\notag&=O_{\Gamma}\left( \frac{1}{\scrP_{\Gamma}(\Gamma g)}T^{2\delGa}(\epsilon T)^{1-2\delGa}\nu_i^{(g^{-1}\Gamma g,i)}\big( \lbrace x\in\RR\,:\, |x|\geq (1+\epsilon)T\rbrace\big)\right)
\\\notag&=O_{\Gamma}\left( \frac{1}{\scrP_{\Gamma}(\Gamma g)}T\epsilon^{1-2\delGa}\nu_i^{(g^{-1}\Gamma g,i)}\big( \lbrace x\in\RR\,:\, |x|\geq T\rbrace\big)\right)
\\\notag&=O_{\Gamma}\left(T^{1-\delGa}\epsilon^{1-2\delGa}  \sYGa(\Gamma g a_T)^{1-\delGa}\right).
\end{align}
\item $\lbrace u\,:\,|u|\geq 2T\rbrace$. For the final integral, we use dyadic decomposition: 
\begin{align*}
\frac{\scrN_{\Gamma}c_{\delGa}(n)}{\scrP_{\Gamma}(\Gamma g)}&\int_{\lbrace u\,:\,2|u|\geq 2T\rbrace}\left( \int_{-u-T}^{T-u} \psi_n(t)\,dt -\kappa_{n}(\delGa)\right)(1+u^2)^{\delGa}\,d\nu_i^{(g^{-1}\Gamma g,i)}(u)
\\&= \frac{\scrN_{\Gamma}c_{\delGa}(n)}{\scrP_{\Gamma}(\Gamma g)}\sum_{m=1}^{\infty}\int_{\lbrace u\,:\,2^mT\leq|u|< 2^{m+1}T\rbrace}\int_{-u-T}^{T-u} \psi_n(t)\,dt(1+u^2)^{\delGa}\,d\nu_i^{(g^{-1}\Gamma g,i)}(u)
\\& =O_{\Gamma}\left( \frac{1}{\scrP_{\Gamma}(\Gamma g)}\sum_{m=1}^{\infty}(T 2^m)^{2\delGa}\int_{\lbrace u\,:\,2^mT\leq|u|< 2^{m+1}T\rbrace}\int_{-u-T}^{T-u} \psi_0(t)\,dt\,d\nu_i^{(g^{-1}\Gamma g,i)}(u)\right).
\end{align*}
For $u$ such that $T 2^m\leq |u| \leq T 2^{m+1}$, $m\geq 1$, we have
\begin{align*}
\int_{-u-T}^{T-u} \psi_0(t)\,dt \leq &\frac{T}{(1+\min_{t\in[-T,T]}|u-t|^2)^{\delGa}}\leq \frac{T}{(T( 2^m-1))^{2\delGa}}\\&\leq T^{1-2\delGa}2^{-2\delGa (m-1)}\ll T^{1-2\delGa}2^{-2\delGa m},
\end{align*}
so
\begin{align*}
\sum_{m=1}^{\infty}&(T 2^m)^{2\delGa}\int_{\lbrace u\,:\,2^mT\leq|u|< 2^{m+1}T\rbrace}\int_{-u-T}^{T-u} \psi_0(t)\,dt\,d\nu_i^{(g^{-1}\Gamma g,i)}(u)\\&=O\left(T \nu_i^{(g^{-1}\Gamma g,i)}\big( \lbrace x\in\RR\,:\, |x|\geq 2T\rbrace\big) \right).
\end{align*}
We use Lemma \ref{Tinftymeasure} again to obtain
\begin{align}\label{intbdd4}
\frac{\scrN_{\Gamma}c_{\delGa}(n)}{\scrP_{\Gamma}(\Gamma g)}&\int_{\lbrace u\,:\,|u|\geq 2T\rbrace}\left( \int_{-u-T}^{T-u} \psi_n(t)\,dt -\kappa_{n}(\delGa)\right)(1+u^2)^{\delGa}\,d\nu_i^{(g^{-1}\Gamma g,i)}(u)
\\\notag&=O_{\Gamma}\left( T^{1-\delGa} \sYGa(\Gamma g a_T)^{1-\delGa}\right).
\end{align}
\end{enumerate}
Combining \eqref{int1bdd}, \eqref{int2bdd}, \eqref{int3bdd}, and \eqref{intbdd4} gives
\begin{align*}
\int_{-T}^T& \phi_n(\Gamma g n_t)\,dt=\mu_{\Gamma g N}^{\mathrm{PS}}(B_T)\scrN_{\Gamma}c_n(\delGa)\kappa_{n}(\delGa)\\&\qquad +O_{\Gamma}\left((T\epsilon)^{1-2\delGa}\mPSGgN(B_{T})+\sYGa(\Gamma g a_T)^{1-\delGa}\big(T^{\delGa}\epsilon^{2\delGa-1}+T^{1-\delGa}\epsilon^{1-2\delGa}  +T^{1-\delGa}\big)\right).
\end{align*}
Now choosing $\epsilon=T^{-\frac{1}{2}}$ completes the proof (this is permitted since $T\geq 4$, and the only requirement placed on $\epsilon$ is $0<\epsilon\leq\frac{1}{2}$).
\end{proof}
\begin{cor}\label{baseeffeqcor}
Let $\Omega\subset \GaG$ be compact. Then
\begin{equation*}
\frac{1}{\mu_{\Gamma g N}^{\mathrm{PS}}(B_T)}\int_{-T}^T \phi_n(\Gamma g n_t)\,dt =\scrN_{\Gamma}c_n(\delGa)\kappa_{n}(\delGa)+O_{\Gamma,\Omega}\left( T^{\frac{1}{2}-\delGa}\right)
\end{equation*}
for all $g\in \Omega\cap\GaG_{\mathrm{rad}}$, $T\gg_{\Omega} 1$, $n\in\ZZ$.
\end{cor}
\begin{proof}
Divide both sides of \eqref{phiequieq} by $\mPSGgN(B_T)$ and apply Corollary \ref{BTCOMP}.
\end{proof}

\section{Effective Equidistribution in the Orthogonal Complement of $\scrH_{\delGa}$}\label{orthsec}
Let $\scrH_1$ denote the orthogonal complement in $L^2(\GaG)$ of $\scrC_{0}$, i.e.
\begin{equation*}
\scrH_1=\left(\bigoplus_{i=1}^{I}\scrC_{i} \right)\oplus L^2(\GaG)_{\mathit{temp}}
\end{equation*}
(cf.\ Proposition \ref{decomp}).
\subsection{Effective equidistribution}
Str\"ombergsson's proof of \cite[Theorem 1]{Andjmd} carries over to our setting of infinite covolume geometrically finite $\Gamma$, giving the following effective equidistribution result for functions in $\scrH_1$:
\begin{thm}\label{StromEfEq}
For all $f\in \scrS^4(\GaG)\cap\scrH_1\cap \scrB_{\alpha}$, $0\leq\alpha<\frac{1}{2}$, $T\gg1$,
\begin{align*}
\frac{1}{2T}\int_{-T}^T f(\Gamma g n_t)\,dt =O_{\Gamma}\Bigg(\|f\|_{\scrS^4(\GaG)}&\left\lbrace\left(\sfrac{\sYGa(\Gamma g a_T)}{T}\right)^{1-s_1}+\left(\sfrac{\sYGa(\Gamma g a_T)}{T}\right)^{\frac{1}{2}}\log^3\left(2+\sfrac{T}{\sYGa(\Gamma g a_T)}\right)\right\rbrace\\&\qquad\qquad\qquad\qquad\qquad\qquad+ \|f\|_{N_{\alpha}}\left(\sfrac{\sYGa(\Gamma g a_T)}{T}\right)^{\frac{1}{2}}\!\Bigg).
\end{align*} 
\end{thm}
\begin{prfdisc*}\label{stromremark}
 It is assumed throughout \cite{Andjmd} that $\Gamma$ is a lattice. However, by following the proofs of \cite[Proposition 3.1 and Theorem 1]{Andjmd}, one obtains the statement above. (The only place in the aforementioned proofs where the fact that $\GaG$ has finite volume is used is \cite[bottom of p. 304]{Andjmd}. We do not claim (or require) as precise a statement as \cite[Theorem 1]{Andjmd}-in particular, we do not distinguish between the cuspidal and non-cuspidal parts of the tempered spectrum. One may thus replace the arguments of \cite{Andjmd} regarding the tempered cuspidal spectrum on \cite[pp. 304-305]{Andjmd} with the treatment of the continuous spectrum given on \cite[pp. 302-303]{Andjmd}.) Indeed, the results of \cite{Andjmd} are based on a representation-theoretic method first developed by Burger in \cite{Burger90} in order to classify the $N$-invariant ergodic Radon measures on $\GaG$ for $\Gamma$ convex-cocompact (possibly of infinite covolume) with $\delGa>\frac{1}{2}$. In \cite{Andjmd}, Str\"ombergsson combined this method with properties of the invariant height function $\sYGa$ to show the effective equidistribution of dense horocycles in any finite-volume $\GaG$. As noted previously, due to the fact that the cusps of geometrically finite hyperbolic surfaces with infinite volume have the same structure as those of finite volume surfaces, their invariant height functions share essentially the same properties, allowing the same treatment to work here.      
\end{prfdisc*}
The following follows from Theorem \ref{StromEfEq} (and Corollary \ref{BTCOMP}) in the same way that Corollary \ref{baseeffeqcor} follows from Theorem \ref{phiequi}:
\begin{cor}\label{Hs1PSeq}
Let $\Omega\subset \GaG$ be compact. Then for all $\Gamma g\in \Omega\cap\GaG_{\mathrm{rad}}$, $f\in \scrS^4(\GaG)\cap\scrH_1\cap \scrB_{\alpha}$, $0\leq \alpha <\frac{1}{2}$, and $T\gg_{\Omega} 1$,
\begin{align*}
\frac{1}{\mPSGgN(B_T)}\int_{-T}^T f(\Gamma g n_t)\,dt =O_{\Gamma,\Omega}\Bigg(&\|f\|_{\scrS^4(\GaG)}\left(\sfrac{\sYGa(\Gamma g a_T)}{T}\right)^{\delGa-s_1}+ \|f\|_{N_{\alpha}}\left(\sfrac{\sYGa(\Gamma g a_T)}{T}\right)^{\delGa-\frac{1}{2}}\\&\quad+\|f\|_{\scrS^4(\GaG)}\left(\sfrac{\sYGa(\Gamma g a_T)}{T}\right)^{\delGa-\frac{1}{2}}\log^3\left(2+\sfrac{T}{\sYGa(\Gamma g a_T)}\right)\Bigg).
\end{align*} 
\end{cor}

\section{Proof of Theorem \ref{mainthm}} \label{thm1proof}
Before proving our main result, Theorem \ref{mainthm}, we first recall the definition of the \emph{Burger-Roblin measure} associated to $N$ on $\GaG$, denoted $m_{\Gamma}^{\mathrm{BR}}$ (and referred to as the BR-measure for short). 
\subsection{The Burger-Roblin measure}\label{BRmsec} Using the Iwasawa decomposition $G=KAN$, we define a left $\Gamma$-invariant (cf.\ \eqref{nuprops}) and right $N$-invariant measure $\widetilde{m}_{\Gamma}^{\mathrm{BR}}$ on $G$ by
\begin{equation*}
\widetilde{m}_{\Gamma}^{\mathrm{BR}}(f)=\int_{KAN} f(ka_yn_x) y^{\delGa-1}\,dx\,dy\,d\nu_i^{(\Gamma,i)}(k\cdot \infty)\qquad \forall f\in C_c(G).
\end{equation*}
We may also express this in terms of the Patterson-Sullivan and Lebesgue densities as follows: firstly, observe that the map
\begin{equation*}
g \mapsto \big([g]^+,[g]^-, \beta_{[g]^+}(i,g\cdot i)\big)
\end{equation*} 
is a bijection from $G$ to $\left((\partial_{\infty}\HH\times \partial_{\infty}\HH)\setminus \lbrace (u,u)\,:\, u\in \partial_{\infty}\HH\rbrace\right) \times \RR$. We may then write the BR-measure as
\begin{equation*}
\widetilde{m}_{\Gamma}^{\mathrm{BR}}(f)=\int_G f(g) e^{\delGa \beta_{[g]^+}(i,g\cdot i)}e^{\beta_{[g]^-}(i,g\cdot i)}\,dm_i([g]^-)\,d\nu_i^{(\Gamma,i)}([g]^+)\,dr\qquad \forall f\in C_c(G),
\end{equation*}
where $r= \beta_{[g]^+}(i,g\cdot i)$. In a similar manner, we define the so-called $\mathrm{BR_*}$-measure $\widetilde{m}_{\Gamma,N}^{\mathrm{BR_*}}$ on $G$ by
\begin{equation*}
\widetilde{m}_{\Gamma}^{\mathrm{BR_*}}(f)=\int_G f(g) e^{\delGa \beta_{[g]^-}(i,g\cdot i)}e^{\beta_{[g]^+}(i,g\cdot i)}\,dm_i([g]^+)\,d\nu_i^{(\Gamma,i)}([g]^-)\,ds\qquad \forall f\in C_c(G),
\end{equation*}
where $s= \beta_{[g]^-}(i,g\cdot i)$. Observe that $\widetilde{m}_{\Gamma}^{\mathrm{BR_*}}$ is right $U$-invariant, where $U$ is the subgroup of $G$ defined by
\begin{equation*}
U=\lbrace n^*_u=\smatr{1}{0}{u}{1}\,:\, u\in\RR\rbrace.
\end{equation*}
The surjective map $\pi:C_c(G)\rightarrow C_c(\GaG)$ given by $[\pi(f)](\Gamma g):=\sum_{\gamma \in \Gamma} f(\gamma g)$ allows us to then define the measure $m_{\Gamma}^{\mathrm{BR}}$ on $\GaG$ by
\begin{equation*}
m_{\Gamma}^{\mathrm{BR}}\big(\pi(f)\big):=\widetilde{m}_{\Gamma}^{\mathrm{BR}}(f)\qquad \forall f\in C_c(G)
\end{equation*} 
(the left $\Gamma$-invariance of $\widetilde{m}^{\mathrm{BR}}_{\Gamma}$ ensures that $m^{\mathrm{BR}}_{\Gamma}$ is well-defined). The measure $m_{\Gamma,N}^{\mathrm{BR_*}}$ is defined in a completely analogous way. Note that both $m_{\Gamma}^{\mathrm{BR}}$ and $m_{\Gamma,N}^{\mathrm{BR_*}}$ are \emph{infinite} measures on $\GaG$.
\subsection*{Proof of Theorem \ref{mainthm}} 
Without loss of generality, we may assume that $ 1- \delGa\leq \alpha < \frac{1}{2}$. We now write $f$ as the orthogonal sum $f=f_0+f_1$, where $f_0\in\scrC_{0}\cap \scrS^4(\GaG)$ and $f_1\in\scrH_1\cap \scrS^4(\GaG)$. By Lemma \ref{Compbdd}, $f_0\in \scrB_{\alpha}$, hence $f_1=f-f_0\in\scrB_{\alpha}$. This allows us to apply Corollary \ref{Hs1PSeq} to $f_1$, which, after noting that $\|f_1\|_{\scrS^4(\GaG)}\leq \|f\|_{\scrS^4(\GaG)}$ and $\|f_1\|_{N^{\alpha}}=\|f-f_0\|_{N^{\alpha}}\leq \|f\|_{N^{\alpha}}+\|f\|_{N^{\alpha}}\ll_{\Gamma} \|f_0\|_{\scrS^4(\GaG)}+\|f\|_{N^{\alpha}}\leq \|f\|_{\scrS^4(\GaG)}+\|f\|_{N^{\alpha}} $, gives
\begin{align}\label{f1int}
\frac{1}{\mPSGgN(B_T)}\int_{-T}^T f(\Gamma g n_t)\,dt=&\frac{1}{\mPSGgN(B_T)}\int_{-T}^T f_0(\Gamma g n_t)\,dt
\\\notag&O_{\Gamma,\Omega}\Bigg(\|f\|_{\scrS^4(\GaG)}\left(\sfrac{\sYGa(\Gamma g a_T)}{T}\right)^{\delGa-s_1}+ \|f\|_{N_{\alpha}}\left(\sfrac{\sYGa(\Gamma g a_T)}{T}\right)^{\delGa-\frac{1}{2}}
\\\notag&\quad\qquad+\|f\|_{\scrS^4(\GaG)}\left(\sfrac{\sYGa(\Gamma g a_T)}{T}\right)^{\delGa-\frac{1}{2}}\log^3\left(2+\sfrac{T}{\sYGa(\Gamma g a_T)}\right)\Bigg).
\end{align}
To complete the proof, it now suffices to prove that
\begin{equation*}
\frac{1}{\mPSGgN(B_T)}\int_{-T}^T f_0(\Gamma g n_t)\,dt=\sfrac{m_{\Gamma}^{\mathrm{BR}}(f)}{m_{\Gamma}^{\mathrm{BMS}}(\GaG)}+O_{\Gamma,\Omega}\big(\|f\|_{\scrS^4(\GaG)} T^{\frac{1}{2}-\delGa}\big). 
\end{equation*}
We observe that $f_0= \sum_{n\in\ZZ} \langle f,\phi_n\rangle_{L^2(\GaG)}\phi_n$. Using Proposition \ref{Yprops} \textit{(1)}, Lemma \ref{Compbdd}, and the bound $|\phi_n(\Gamma h)|\ll \phi_0(\Gamma h)$, we have
\begin{align*}
 \sum_{n\in\ZZ} |\langle f,\phi_n\rangle_{L^2(\GaG)}\phi_n(\Gamma g n_t)|\ll &\sYGa(\Gamma g n_t)^{1-\delGa}\sum_{n\in\ZZ} |\langle f,\phi_n\rangle_{L^2(\GaG)}|\\&\ll (1+T)^{2-2\delGa}\sYGa(g)^{1-\delGa} \|f\|_{\scrS^1(\GaG)}\qquad \forall t\in B_T, \Gamma g\in \GaG.
\end{align*}
This permits us to write $\int_{-T}^T f_0(\Gamma g n_t)\,dt=\sum_{n\in\ZZ}\langle f,\phi_n\rangle_{L^2(\GaG)} \int_{-T}^T \phi_n(\Gamma g n_t)\,dt$, and so 
Corollary \ref{baseeffeqcor} gives
\begin{align}\label{f0int}
\frac{1}{\mPSGgN(B_T)}&\int_{-T}^T f_0(\Gamma g n_t)\,dt=\sum_{n\in\ZZ} \langle f,\phi_n\rangle_{L^2(\GaG)} \int_{-T}^T \phi_n(\Gamma g n_t)\,dt
\\\notag&=\sum_{n\in\ZZ} \langle f,\phi_n\rangle_{L^2(\GaG)} \left( \scrN_{\Gamma}c_n(\delGa)\kappa_{n}(\delGa)+O_{\Gamma,\Omega}\left( T^{\frac{1}{2}-\delGa}\right)\right)
\\\notag&=\left(\sum_{n\in\ZZ}  \scrN_{\Gamma}c_n(\delGa)\kappa_{n}(\delGa) \langle f,\phi_n\rangle_{L^2(\GaG)} \right) +O_{\Gamma,\Omega}\left( T^{\frac{1}{2}-\delGa}\left( \sum_{n\in\ZZ}|\langle f,\phi_n\rangle_{L^2(\GaG)} |\right)\right)
\\\notag&=\scrM_{\Gamma}(f)+O_{\Gamma,\Omega}\left( T^{\frac{1}{2}-\delGa}\|f\|_{\scrS^1(\GaG)}\right)
\end{align}
(cf.\ \eqref{Mdef}).

Now, \eqref{f1int} and \eqref{f0int} show that $ \lim_{T\rightarrow \infty} \frac{1}{\mPSGgN(B_T)}\int_{-T}^T f_0(\Gamma g n_t)\,dt=\scrM_{\Gamma}(f)$. However, \cite[Theorem 1.5]{MoOh} or \cite[Theorem 1.1]{MauSchap} gives $\frac{1}{\mPSGgN(B_T)}\int_{-T}^T \psi(\Gamma g n_t)\,dt=\sfrac{m_{\Gamma}^{\mathrm{BR}}(\psi)}{m_{\Gamma}^{\mathrm{BMS}}(\GaG)}$ for all $\psi\in C_c(\GaG)$ (note that \emph{both} $\mPSGgN$ and $m_{\Gamma}^{\mathrm{BR}}$ are scaled with a factor $\scrN_{\Gamma}$ compared with those of \cite{MoOh}-this enables us to use the cited result). Observing that $|\scrM_{\Gamma}(f)|\ll_{\Gamma} \|f\|_{\scrS^1(\GaG)}$, we obtain the claimed extension of $f\mapsto m_{\Gamma}^{\mathrm{BR}}(f)$.

\hspace{425.5pt}\qedsymbol

\begin{remark}
Since $C_c^{\infty}(\GaG)\subset L^1(\GaG,m_{\Gamma}^{\mathrm{BR}})\cap\scrS^1(\GaG)$, we obtain the following identity for the BR-measure:
\begin{equation}\label{mBRident}
m_{\Gamma}^{\mathrm{BR}}(f)=m_{\Gamma}^{\mathrm{BMS}}(\GaG)\scrM_{\Gamma}(f)=\sum_{n\in \ZZ} \frac{m_{\Gamma}^{\mathrm{BMS}}(\GaG)\scrN_{\Gamma}\kappa_0(\delGa)}{c_n(\delGa)} \langle f, \phi_n\rangle_{L^2(\GaG)}\quad \forall f\in C_c^{\infty}(\GaG).
\end{equation}
A similar identity is obtained in \cite[Theorem 7.3]{LeeOh}. At a first glance, our formula appears to be different from that given in \cite{LeeOh}; the identities do not appear to give the same value even up to scaling. A closer inspection reveals that this is due to a small typo in \cite{LeeOh}: in the case $n=2$, the formula given in \cite[Theorem 4.6]{LeeOh} should read
\begin{equation*}
\phi_{\mathit{l}}^N(a_y)=c_2(0)\sfrac{\sqrt{\Gamma(\delta)\Gamma(1-\delta+l)}}{\sqrt{\Gamma(1-\delta)\Gamma(\delta+l)}}y^{1-\delta}.
\end{equation*}
After making a subsequent correction to \cite[(6.1), p.\ 610]{LeeOh}, it is straightforward to verify that \eqref{mBRident} agrees with \cite[Theorem 7.3]{LeeOh} (at least up to scaling).
\end{remark}
\section{Convex-Cocompact $\GaG$}\label{deltalessthan1/2}
We will now restrict our attention to convex cocompact $\Gamma$ and demonstrate how one can deduce effective equidistribution of non-closed horocycles from the exponential mixing of the diagonal action with respect to the \emph{Bowen-Margulis-Sullivan measure} (abbreviated as the BMS-measure) \emph{without the assumption that} $\delGa>\frac{1}{2}$. As such, throughout this section $\Gamma$ is non-elementary and convex cocompact. As previously noted, if $\delGa\leq \frac{1}{2}$ then $\Gamma$ is necessarily convex cocompact. 

\subsection{Exponential mixing}
The key result which we need is \emph{exponential mixing} of the diagonal subgroup of $G$. This was first obtained by Stoyanov with respect to the BMS-measure for convex cocompact $\Gamma$ \cite{Stoyanov}. In \cite[Section 5.2]{OhWinter}, Oh and Winter show how to obtain an exponential mixing statement for the Haar measure from that for the BMS-measure. It is this result that will be the main ingredient in the proof of Theorem \ref{mainthm2}.

Before giving the precise statement, we recall some of the terminology introduced in Section \ref{intro}: for $\Omega\subset \GaG$, we let $\scrS^m(\Omega)$ denote the closure of $\lbrace f\in C_c^{\infty}(\Omega)\,:\, f|_{\partial\Omega}=0\rbrace$ with respect to the norm $\|\cdot\|_{\scrS^m(\GaG)}$. Similarly, we let $\|\cdot\|_{W^m}$ denote the standard $L^2$-Sobolev norm of order $m$ on $\RR$, and for an interval $I\subset \RR$, we let $W^m(I)$ denote the closure of $\lbrace C^{\infty}_c(I)\,:\, f|_{\partial I}=0\rbrace$ with respect to $\|\cdot\|_{W^m}$.

Combining \cite[Corollary 1.5]{Stoyanov} with \cite[Theorem 5.8]{OhWinter} gives

\begin{thm}\label{expmixing}
There exists $\eta_0>0$ such that for any compact subset $\Omega\subset \GaG$,
\begin{equation*}
\int_{\GaG}\!\! f_1(\Gamma g a_y) f_2(\Gamma g)\,d\mu_{\GaG}(\Gamma g)\!=\!\sfrac{m_{\Gamma}^{\mathrm{BR}}(f_1)m_{\Gamma}^{\mathrm{BR_*}}(f_2)}{m_{\Gamma}^{\mathrm{BMS}}(\GaG)}y^{1-\delGa} \!+O_{\Gamma,\Omega}\!\left(\!y^{1-\delGa +\eta_0}\|f_1\|_{\scrS^3(\GaG)}\|f_2\|_{\scrS^3(\GaG)} \!\right)
\end{equation*}  
for all $0<y\leq 1$, $f_1,\,f_2\in \scrS^3(\Omega)$.
\end{thm}
\begin{remark}
Observe that $y\rightarrow\infty$ in \cite[Theorem 5.8]{OhWinter}. Using the $G$-invariance of $\mu_{\GaG}$ and the fact that our definitions of $m_{\Gamma}^{\mathrm{BR}}$ and $m_{\Gamma}^{\mathrm{BR_*}}$ are interchanged compared with those in \cite{OhWinter}, we obtain the main term stated here. To obtain our error term from that of \cite[Theorem 5.8]{OhWinter}, we simply use the Sobolev inequality $\|f\|_{C^1}\ll_{\Gamma} \|f\|_{\scrS^3(\GaG)}$ (cf.\ Lemma \ref{Sobbdd}).
\end{remark}

\subsection{Effective equidistribution of expanding translates}\label{translatessec} Since $\Gamma$ is convex cocompact, there is a uniform lower bound on the injectivity radius at each point of $\GaG$. This allows us to deduce the effective equidistribution of non-closed horocycles from the effective equidistribution of \emph{expanding translates} of compact pieces of horospherical orbits. This result in turn follows from the exponential mixing of the diagonal subgroup via the classical ``Margulis thickening trick" see e.g.\ Kleinbock and Margulis \cite[Proposition 2.4.8]{KleinMarg} for the proof in the general finite-volume setting. 

For infinite volume $\GaG$, the result we require is due to Mohammadi and Oh \cite[Theorem 5.13]{MohammadiOh}. The main complication compared with the finite volume setting is that the Lebesgue and Haar measures can (in general) give much greater mass to subsets than those given by the PS- and BR-measures. One must thus avoid bounding any approximations of functions until after making use of the exponential mixing from Theorem \ref{expmixing}. Since there are slight variations in our notation and setting compared with \cite{MohammadiOh} (as well as the fact that we will also require similar estimates in the proof of Theorem \ref{mainthm2}), we closely follow \cite[Section 5]{MohammadiOh} and reproduce the key steps of their proof. We refer the reader to \cite[Section 5]{MohammadiOh} for more details.

We start by recalling the Bruhat $NAU$ decomposition of $G$: $NAU$ is an open neighbourhood of the identity in $G$ and $G=\overline{NAU}$ (cf.\ \cite[Proposition 8.45]{Knapp2}). This allows us to make the following decomposition of the $\mathrm{BR_*}$-measure (cf.\ \cite[(5.3), p.\ 868]{MohammadiOh}):
\begin{lem}\label{measdecomp}
Let $B_1\subset N$, $B_2\subset A$, $B_3\subset U$ be open neighbourhoods of the identity (in the respective subgroups) and let $g\in G$. Then for any $f\in C_c(G)$ with $\mathrm{supp}(f)\subset g B_1B_2B_3$,
\begin{equation*}
\widetilde{m}^{\mathrm{BR_*}}_{\Gamma}(f)=\int_{\lbrace n_x\in B_1\rbrace}\int_{\lbrace a_y\in B_2\rbrace}\int_{\lbrace n^*_u\in B_3\rbrace}f(g n_x a_y n^*_u)\,e^{\delGa\beta_{[gn_x]^-}(i,gn_x\cdot i)} y^{1-\delGa}\, du\,dy\,d\nu_i^{(\Gamma,i)}([gn_x]^-).
\end{equation*}   
\end{lem}
\begin{proof}
Using the definition from Section \ref{BRmsec}:
\begin{align*}
\widetilde{m}_{\Gamma}^{\mathrm{BR_*}}(f)=\int_G f(h) e^{\delGa s}e^{\beta_{[h]^+}(i,h\cdot i)}\,ds\,dm_i([h]^+)\,d\nu_i^{(\Gamma,i)}([h]^-)
\end{align*}
where $s= \beta_{[h]^-}(i,h\cdot i)$. Writing $h=gn_xa_yn^*_u$, we observe that 
\begin{align*}
&[gn_x a_y n^*_u]^-=[gn_x]^-
\\&s=\beta_{[gn_x]^-}(i,gn_x a_y n_u^*\cdot_i)=\beta_{[gn_x]^-}(i,gn_x\cdot_i)+\beta_{0}(i,a_y n^*_u\cdot_i)=\beta_{[gn_x]^-}(i,gn_x\cdot_i)-\log y
\\&e^{\beta_{[gn_xa_yn^*_u]^+}(i,gn_xa_yn^*_u\cdot i)}\,dm_i([gn_xa_yn^*_u]^+)=e^{\beta_{\frac{1}{u}}((gn_xa_y)^{-1}\cdot i,n^*_u\cdot i)}\,dm_{(gn_x a_y)^{-1}\cdot i}(\sfrac{1}{u})
\\&\quad=e^{\beta_{\frac{1}{u}}((gn_xa_y)^{-1}\cdot i,n^*_u\cdot i)}e^{-\beta_{\frac{1}{u}}((gn_xa_y)^{-1}\cdot i, i)}\,dm_{ i}(\sfrac{1}{u})=e^{\beta_{\frac{1}{u}}(i,n_{u}^*\cdot i)}\,dm_{ i}(\sfrac{1}{u})
\\&\quad=(u^2+1)\frac{d(\sfrac{1}{u})}{1+\frac{1}{u^2}}.
\end{align*}
This gives $e^{\beta_{\frac{1}{u}}(i,n_{u}^*\cdot i)}\,dm_{ i}(\sfrac{1}{u}) \,ds= y \,du\,dy$, and so
\begin{align*}
\widetilde{m}_{\Gamma}^{\mathrm{BR_*}}(f)=&\!\!\int\limits_{\lbrace g n_x a_y n^*_u\in gB_1B_2B_3\rbrace}\!\!\! \!\!\!\!f(g n_x a_y n^*_u)\,e^{\delGa\beta_{[gn_x]^-}(i,gn_x\cdot i)}y^{-\delGa}e^{\beta_{\frac{1}{u}}(i,n_{u}^*\cdot i)}\,dm_{ i}(\sfrac{1}{u})\, ds\,d\nu_i^{(\Gamma,i)}([gn_x]^-)
\\&=\iiint\limits_{ B_1 B_2 B_3}f(g n_x a_y n^*_u)\,e^{\delGa\beta_{[gn_x]^-}(i,gn_x\cdot i)}y^{-\delGa}\, du\, y\,dy\,d\nu_i^{(\Gamma,i)}([gn_x]^-).
\end{align*}
\end{proof}
Let $\mathsf{dist}_G$ denote the Riemannian metric on $G$ induced from the Killing form on $\fg$ and $\scrB_r$ to denote the open ball of radius $r$ around the identity in $G$. The corresponding norm on $\fg$ is denoted by $|\cdot|$. We now choose $r_{\Gamma}\leq 1$ small enough so that the exponential map is surjective onto $\scrB_{r_{\Gamma}}$ and for each $\Gamma g\in \GaG$, the map from $\scrB_{r_{\Gamma}}$ to $\GaG$ given by $h\mapsto \Gamma g h$ is injective. 
\begin{lem} \label{diffbdd}
\begin{equation*}
|f(\Gamma g h)-f(\Gamma g)|\ll_{\Gamma} r \|f\|_{\scrS^3(\GaG)}\qquad \forall 0\leq r \leq r_{\Gamma},\; g\in G,\; h\in \scrB_{r},\; f\in \scrS^3(\GaG).
\end{equation*}
\end{lem}
\begin{proof}
Given $h$ in such a $\scrB_r$, there exists $X\in\fg$ such that $h=\exp(X)$ and $|X|\ll r$. We then have
\begin{equation*}
|f(\Gamma g h)-f(\Gamma g)|\leq\int_0^1 |Xf\big(\Gamma g\exp(sX)\big)|\,ds\ll_{\Gamma} \|Xf\|_{\scrS^2(\GaG)}\ll r\|f\|_{\scrS^3(\GaG)}.
\end{equation*}
\end{proof}
We also let $\epsilon_{\Gamma}\leq r_{\Gamma}$ be small enough so that
\begin{equation*}
\left \lbrace n_xa_yn^*_u\,:\, \max\lbrace |x|, |\log y|, |u|\rbrace < \epsilon_{\Gamma}\right\rbrace \subset \scrB_{r_{\Gamma}/2}. 
\end{equation*}
\begin{thm}\label{translates}
There exists $\eta_1>0$ such that for any compact subset $\Omega\subset \GaG$, 
\begin{equation*}
\int_{-\epsilon_{\Gamma}}^{\epsilon_{\Gamma}} f(\Gamma g n_t a_y) \phi(t)\,dt=\!\sfrac{m_{\Gamma}^{\mathrm{BR}}(f)\mu_{\Gamma gN}^{\mathrm{PS}}(\phi)}{m_{\Gamma}^{\mathrm{BMS}}(\GaG)}y^{1-\delGa} \!\!+\!O_{\Gamma,\Omega}\!\left(\!y^{1-\delGa +\eta_1}\|f\|_{\scrS^3(\GaG)}\lbrace\|\phi\|_{W^3}\!+\!\mu_{\Gamma gN}^{\mathrm{PS}}(\phi)\rbrace \!\right)
\end{equation*}
for all $\Gamma g \in \Omega$, and non-negative $f\in \scrS^3(\Omega)$, $\phi\in C_c^{\infty}\big((-\epsilon_{\Gamma},\epsilon_{\Gamma})\big)$.
\end{thm}
\begin{remark}
We have previously only defined the measures $m^{\mathrm{PS}}_{\Gamma g N}$ for radial points $\Gamma g$. While we will only need Theorem \ref{translates} for the radial points, we note that since $\Gamma$ is convex-cocompact, the map from $N$ to $\GaG$ given by $n\mapsto \Gamma g n$ is injective for all $\Gamma g\in \GaG$; the definition given in Section \ref{PS-N} therefore still works for all $\Gamma g\in \GaG$. It is in the case that $\Gamma$ is not convex-cocompact that more care is required in the definition; this is due to the presence of periodic horocycles around the cusps of $\GaG$, cf.\ \cite[Section 2]{MoOh}. 
\end{remark}

\begin{proof}
We start by defining, for $\epsilon\leq \epsilon_{\Gamma}$, functions $f_{\epsilon}^+$ and $f_{\epsilon}^-$ by
\begin{equation*}
f_{\epsilon}^+(\Gamma g):= \sup_{h\in \scrB_{\epsilon}} f(\Gamma g h), \qquad f_{\epsilon}^{-}(\Gamma g):= \inf_{h\in \scrB_{\epsilon}} f(\Gamma g h).
\end{equation*}
Observe that $f_{\epsilon}^{\pm}\in \scrS^3(\Omega \scrB_{\epsilon})$ and by Lemma \ref{diffbdd}, $|f(\Gamma g)- f_{\epsilon}^{\pm}(\Gamma g)|\ll_{\Gamma} \epsilon \|f\|_{\scrS^3(\GaG)}$.

By \cite[Lemma 2.4.7]{KleinMarg}, given $\epsilon>0$, there exists $\rho_\epsilon\in \scrB_{\epsilon}\cap C_c^{\infty}(AU)$ such that:
\begin{equation*}
\rho_{\epsilon}(a_yn^*_u)\geq 0 \; \forall v\in\RR_{>0}\;u\in\RR, \quad \int_{\RR_{>0}}\int_{\RR}\rho_{\epsilon}(a_vn^*_u)\,\frac{ du\,dv}{v^2}=1.
\end{equation*}
We now define a function $\Phi_{\epsilon}\in C_c^{\infty}(\GaG)$ by
\begin{equation*}
\Phi_{\epsilon}(\Gamma h)=\begin{cases} \phi(t)\rho_{\epsilon}(a_vn^*_u)\qquad&\mathrm{if}\; \Gamma h = \Gamma g n_t a_v n^*_u\\
0\qquad&\mathrm{otherwise.} 
\end{cases}
\end{equation*}
Observe that $\Phi_{\epsilon}$ is well-defined is due to the uniqueness of the $NAU$ decomposition, and that $\epsilon\leq \epsilon_{\Gamma}\leq \frac{r_{\Gamma}}{2}$ (which is less than the injectivity radius of $\GaG$); $\Phi_{\epsilon}$ is thus supported on $\Gamma g \scrB_{\epsilon_{\Gamma}}\subset \Omega \scrB_{\epsilon_{\Gamma}}$. 
Using this definition, we have
\begin{align*}
\int_{-\epsilon_{\Gamma}}^{\epsilon_{\Gamma}} f(\Gamma g n_t a_y) \phi(t)\,dt&=\int_{-\epsilon_{\Gamma}}^{\epsilon_{\Gamma}} f(\Gamma g n_t a_y) \phi(t)\,dt\left(\int_{\RR_{>0}}\int_{\RR}\rho_{\epsilon}(a_vn^*_u)\,\sfrac{ du\,dv}{v^2} \right)
\\=&\int_{-\epsilon_{\Gamma}}^{\epsilon_{\Gamma}} \int_{\RR_{>0}}\int_{\RR}f(\Gamma g n_t a_y) \phi(t)\rho_{\epsilon}(a_v n^*_u)\,\sfrac{ dt\,dv\,du}{v^2}
\\=&\int_{-\epsilon_{\Gamma}}^{\epsilon_{\Gamma}} \int_{\RR_{>0}}\int_{\RR}f\big(\Gamma g n_ta_v n^*_u a_y(a_v  n^*_{yu})^{-1} \big) \Phi_{\epsilon}(\Gamma g n_t a_v n^*_u)\,\sfrac{ dt\,dv\,du}{v^2}.
\end{align*}
Since $y\leq 1$ and $a_vn_u^*\in \scrB_{\epsilon}$, $(a_v n^*_{yu})^{-1}\in\scrB_{\epsilon}$, hence
\begin{equation*}
f_{\epsilon}^{-} (\Gamma g n_t a_v n^*_u a_y)\leq f\big(\Gamma g n_ta_v n^*_u a_y(a_v  n^*_{yu})^{-1} \big)\leq f_{\epsilon}^{+} (\Gamma g n_t a_v n^*_u a_y).
\end{equation*}
Now, $d\mu_G(n_x a_v n^*_u)=\frac{dx\,dv\,du}{v^2}$; we may thus bound the integral we are concerned with as follows:
\begin{equation*}
\int_{\scrB_{\epsilon_{\Gamma}}} 
f_{\epsilon}^-(\Gamma h a_y)\Phi_{\epsilon}(\Gamma h)\,d\mu_G(h)\leq \int_{-\epsilon_{\Gamma}}^{\epsilon_{\Gamma}} f(\Gamma g n_t a_y) \phi(t)\,dt \leq \int_{\scrB_{\epsilon_{\Gamma}}} 
f_{\epsilon}^+(\Gamma h a_y)\Phi_{\epsilon}(\Gamma h)\,d\mu_G(h).
\end{equation*}
By Theorem \ref{expmixing}:
\begin{align*}
\int_{\scrB_{\epsilon_{\Gamma}}} 
f_{\epsilon}^{\pm}(\Gamma h a_y)\Phi_{\epsilon}(\Gamma h)\,d\mu_G(h)&=\int_{\Gamma g\scrB_{\epsilon_{\Gamma}}}f_{\epsilon}^{\pm}(\Gamma h a_y)\Phi_{\epsilon}(\Gamma h)\,d\mu_{\GaG}(\Gamma h)
\\=&\int_{\GaG}f_{\epsilon}^{\pm}(\Gamma h a_y)\Phi_{\epsilon}(\Gamma h)\,d\mu_{\GaG}(\Gamma h)
\\=\sfrac{m_{\Gamma}^{\mathrm{BR}}(f_{\epsilon}^{\pm})m_{\Gamma}^{\mathrm{BR_*}}(\Phi_{\epsilon})}{m_{\Gamma}^{\mathrm{BMS}}(\GaG)}&y^{1-\delGa} \!+O_{\Gamma,\Omega\scrB_{\epsilon_{\Gamma}}}\!\left(\!y^{1-\delGa +\eta_0}\|f_{\epsilon}^{\pm}\|_{\scrS^3(\GaG)}\|\Phi_{\epsilon}\|_{\scrS^3(\GaG)} \!\right).
\end{align*}
We have $\|f_{\epsilon}^{\pm}\|_{\scrS^3(\GaG)}\ll_{\Gamma,\Omega} \|f\|_{\scrS^3(\GaG)}$ (cf.\ \cite[(5.8), p.\ 868]{MohammadiOh}). Also, again appealing to \cite[Lemma 2.4.7]{KleinMarg} gives the bound $\|\Phi_{\epsilon}\|_{\scrS^3(\GaG)}\ll \|\phi\|_{W^3}\epsilon^{-4}$, hence
\begin{equation}\label{Obdd}
O_{\Gamma,\Omega\scrB_{\epsilon_{\Gamma}}}\!\left(\!y^{1-\delGa +\eta_1}\|f_{\epsilon}^{\pm}\|_{\scrS^3(\GaG)}\|\Phi_{\epsilon}\|_{\scrS^3(\GaG)} \!\right)=O_{\Gamma,\Omega}\!\left(\!y^{1-\delGa +\eta_0}\|f\|_{\scrS^3(\GaG)}\|\phi\|_{W^3}\epsilon^{-4} \!\right).
\end{equation} 
Since $m^{\mathrm{BR}}_{\Gamma}$ is locally finite, Lemma \ref{diffbdd} gives
\begin{equation}\label{mfepsBRbdd}
m_{\Gamma}^{\mathrm{BR}}(f_{\epsilon}^{\pm})=m_{\Gamma}^{\mathrm{BR}}(f)+m_{\Gamma}^{\mathrm{BR}}(f_{\epsilon}^{\pm}-f)=m_{\Gamma}^{\mathrm{BR}}(f)+O_{\Gamma,\Omega}(\epsilon \|f\|_{\scrS^3(\GaG)}).
\end{equation}
We now use Lemma \ref{mBRident} to compute $m_{\Gamma}^{\mathrm{BR_*}}(\Phi_{\epsilon})$:
\begin{align*}
m_{\Gamma}^{\mathrm{BR_*}}(\Phi_{\epsilon})=&\widetilde{m}^{\mathrm{BR_*}}\big([h\mapsto\Phi_{\epsilon}(\Gamma g h)]|_{\scrB_{\epsilon_{\Gamma}}}\big) 
\\&=\int_{-\epsilon_{\Gamma}}^{\epsilon_{\Gamma}}\iint\limits_{\;\;\lbrace a_vn_u^*\in \scrB_\epsilon\rbrace}\phi(t)\rho_{\epsilon}(a_yn^*_u) \,e^{\delGa\beta_{[gn_x]^-}(i,gn_x\cdot i)}v^{1-\delGa}\, du\, \,dv\,d\nu_i^{(\Gamma,i)}([gn_x]^-)
\\&=\int_{-\epsilon_{\Gamma}}^{\epsilon_{\Gamma}}\iint\limits_{\;\;\lbrace a_vn_u^*\in \scrB_\epsilon\rbrace}\phi(t)\rho_{\epsilon}(a_yn^*_u) \,e^{\delGa\beta_{[gn_x]^-}(i,gn_x\cdot i)}v^{3-\delGa}\, \sfrac{du\, \,dv}{v^2}\,d\nu_i^{(\Gamma,i)}([gn_x]^-).
\end{align*}
For $v\in\scrB_{\epsilon}$, $|\log v|\ll \epsilon$, so $v^{3-\delGa}=1+O_{\Gamma}(\epsilon)$, hence
\begin{align*}
m_{\Gamma}^{\mathrm{BR_*}}(\Phi_{\epsilon})=& \mu_{\Gamma gN}^{\mathrm{PS}}(\phi)\times \!\!\!\!\!\!\iint\limits_{\;\;\lbrace a_vn_u^*\in \scrB_\epsilon\rbrace}\!\!\!\!\!\!\rho_{\epsilon}(a_yn^*_u) \big(1+O_{\Gamma}(\epsilon)\big)\, \sfrac{du\, \,dv}{v^2}
\notag\\&=\mu_{\Gamma gN}^{\mathrm{PS}}(\phi)\big(1+O_{\Gamma}(\epsilon)\big).
\end{align*}
This, together with \eqref{mfepsBRbdd}, gives
\begin{align*}
m_{\Gamma}^{\mathrm{BR}}(f_{\epsilon}^{\pm})m_{\Gamma}^{\mathrm{BR_*}}(\Phi_{\epsilon})&=\left(m_{\Gamma}^{\mathrm{BR}}(f)+O_{\Gamma,\Omega}(\epsilon \|f\|_{\scrS^3(\GaG)})\right)\mu_{\Gamma gN}^{\mathrm{PS}}(\phi)\big(1+O_{\Gamma}(\epsilon)\big)
\\&=m_{\Gamma}^{\mathrm{BR}}(f)\mu_{\Gamma gN}^{\mathrm{PS}}(\phi)+O_{\Gamma,\Omega}\left( \epsilon m_{\Gamma}^{\mathrm{BR}}(f)\mu_{\Gamma gN}^{\mathrm{PS}}(\phi) + \epsilon  \|f\|_{\scrS^3(\GaG)}\mu_{\Gamma gN}^{\mathrm{PS}}(\phi)\right)
\\&=m_{\Gamma}^{\mathrm{BR}}(f)\mu_{\Gamma gN}^{\mathrm{PS}}(\phi)+O_{\Gamma,\Omega}\left(\epsilon \mu_{\Gamma gN}^{\mathrm{PS}}(\phi)\|f\|_{\scrS^3(\GaG)} \right).
\end{align*}
Combining this expression with \eqref{Obdd} yields
\begin{align*}
\int_{\scrB_{\epsilon_{\Gamma}}} 
f_{\epsilon}^{\pm}(\Gamma h a_y)\Phi_{\epsilon}(\Gamma h)\,d\mu_G(h)= &\sfrac{m_{\Gamma}^{\mathrm{BR}}(f)\mu_{\Gamma gN}^{\mathrm{PS}}(\phi)}{m_{\Gamma}^{\mathrm{BMS}}(\GaG)}y^{1-\delGa}
\\& +O_{\Gamma,\Omega}\left( y^{1-\delGa}\|f\|_{\scrS^3(\GaG)} \left\lbrace\epsilon \mu_{\Gamma gN}^{\mathrm{PS}}(\phi)+ y^{\eta_0}\epsilon^{-4}\|\phi\|_{W^3}\right\rbrace\right).
\end{align*}
Since $\int_{-\epsilon_{\Gamma}}^{\epsilon_{\Gamma}} f(\Gamma g n_t a_y) \phi(t)\,dt$ is bounded from above and below by the integrals in the right-hand side of this expression, the same must hold for it. Choosing $\epsilon=y^{\frac{\eta_0}{5}}$ then completes the proof, with $\eta_1=\frac{\eta_0}{5}$.
\end{proof} 

\subsection{The shadow lemma} The final step before proceeding with the proof of Theorem \ref{mainthm2} involves adapting the results of Section \ref{Shadowsec} to the case $\delGa\leq \frac{1}{2}$. For $\delGa \leq \frac{1}{2}$, the integral in \eqref{phi0def} still defines an eigenfunction of $-\Delta$ on $\GaH$ with eigenvalue $\delGa(1-\delGa)$ (cf.\ \cite{Patterson1,Patterson2}), however it is no longer in $L^2(\GaH)$; we thus define
\begin{equation*}
\widetilde{\phi}_0(\Gamma g):=\int_{\partial_{\infty}\HH} e^{-\delGa \beta_u(g\cdot i,i)}\,d\nu_i^{(\Gamma,i)}(u)
\end{equation*}
(i.e.\ we remove the constant $\scrN_{\Gamma}$ from the definition given in \eqref{phi0def} since it is not well-defined for $\delGa\leq \frac{1}{2}$). We note, however, that $\widetilde{\phi}_0$ is bounded:
\begin{lem}\label{bddlem}
Let $\Gamma$ be convex cocompact. Then $\widetilde{\phi}_0\in L^{\infty}(\GaG)$.
\end{lem} 
In fact, $\widetilde{\phi}_0$ decays outside the convex core of $\GaG$, cf.\ \cite[Proposition 4.2]{BurgerCanary}, though for $\delGa\leq\frac{1}{2}$ not fast enough so that $\widetilde{\phi}_0\in L^2(\GaG)$.

Since $\widetilde{\phi}_0(\Gamma g)\ll_{\Gamma} 1=\sYGa(\Gamma g)$, the results of Section \ref{Shadowsec} all hold even without the assumption $\delGa>\frac{1}{2}$. Moreover, simplifications occur due to the fact that we no longer have to take $\sYGa$ into account. Lemmas \ref{shadlem}, \ref{Tinftymeasure}, and \ref{diffmeas} in the convex cocompact setting read as follows:
\begin{lem}\label{shadlem1}
For all $z,\,w\in\HH$, $r>0$,
\begin{equation*}
\nu_z^{(\Gamma,i)}\big(\scrO_z(w,r)\big)\ll_{\Gamma} e^{2\delGa r-\delGa\mathrm{dist}(z,w)}. 
\end{equation*}
\end{lem}
\begin{lem}\label{Tinftymeasure1}
\begin{equation*}
\nu_i^{(g^{-1}\Gamma g,i)}\big( \lbrace x\in\RR\,:\, |x|\geq T\rbrace\big)\ll_{\Gamma}\scrP_{\Gamma}(\Gamma g) T^{-\delGa} \qquad \forall g\in G,\,T\geq 1.
\end{equation*}
\end{lem}

\begin{lem}\label{diffmeas1}
\begin{align*}
&\nu_i^{(g^{-1}\Gamma g,i)}\big(\lbrace u\in\RR\,:\,(1-\epsilon)T\leq |u|\leq (1+\epsilon)T\rbrace \big)\ll_{\Gamma} \scrP_{\Gamma}(\Gamma g)\epsilon^{\delGa} T^{-\delGa}
\end{align*}
for all $g\in G$, $T\geq 2$, $0\leq\epsilon\leq\frac{1}{2}$. 
\end{lem}
Noting that Theorem \ref{Shadow} also holds for convex cocompact $\Gamma$ without the assumption $\delGa>\frac{1}{2}$, cf., e.g.., \cite[Theorem 4.6.2]{Nicholls}.   Proposition \ref{muPSbd} thus also holds, as well as Corollary \ref{BTCOMP}, which in the current setting reads as
\begin{cor}\label{BTCOMP1}
Let $\Omega\subset \GaG$ be compact. Then
\begin{equation*}
\mu_{\Gamma g N}^{\mathrm{PS}}(B_T)\gg_{\Omega}  T^{\delGa}  \qquad \forall \,\Gamma g \in \Omega\cap \GaG_{\mathrm{rad}},\;T\gg_{\Omega} 1.
\end{equation*}
\end{cor}

\subsection{Proof of Theorem \ref{mainthm2}}
We start by assuming that $f$ is $\RR_{\geq 0}$-valued. For $r>0$, we have 
\begin{equation}\label{Tident}
\int_{-T}^T \!f(\Gamma g n_t)\,dt=\frac{T}{r}\int_{-r}^{r} \!f\big(\Gamma g n_{T t/r}\big)\,dt=\frac{T}{r}\int_{-r}^{r} \!f\big(\Gamma ga_{T/r} n_{ t}a_{T/r}^{-1}\big)\,dt.
\end{equation}
By \cite[Lemma 2.4.7]{KleinMarg}, given $\epsilon>0$, there exists $\psi_\epsilon\in C_c^{\infty}\big((-\epsilon,\epsilon)\big)$ such that:
\begin{equation*}
\psi_{\epsilon}(x)\geq 0 \; \forall x\in\RR, \quad \int_{\RR}\psi_{\epsilon}(x)\,dx=1,\quad \|\psi_{\epsilon}\|_{W^3}\ll \epsilon^{-\frac{7}{2}}.
\end{equation*}
For $\epsilon<\epsilon_{\Gamma}$ ($\epsilon_{\Gamma}$ being as in Section \ref{translatessec}), let $\chi_{\epsilon}=\psi_{\epsilon/2}\ast\mathbbm{1}_{[-\epsilon_{\Gamma}+\epsilon/2,\epsilon_{\Gamma}-\epsilon/2]}$, i.e.
\begin{equation*}
\chi_{\epsilon}(x)=\int_{-\epsilon_{\Gamma}+\frac{\epsilon}{2}}^{\epsilon_{\Gamma}-\frac{\epsilon}{2}}\psi_{\epsilon}(x-u)\,du.
\end{equation*}
Observe that $0\leq \chi_{\epsilon}(x)\leq 1$ and
\begin{equation*}
\chi_{\epsilon}(x)=\begin{cases} 1\qquad&\mathrm{if\;} |x|\leq \epsilon_{\Gamma}-\epsilon\\0\qquad&\mathrm{if\;} |x|\geq \epsilon_{\Gamma}\end{cases}
\end{equation*}
for all $x\in\RR$. Note also that 
\begin{align*}
\int_{\RR} |\sfrac{d^j}{dx^j}\chi_{\epsilon}(x)|^2\,dx=\int_{-\epsilon_{\Gamma}}^{\epsilon_{\Gamma}} \left|\int_{-\epsilon_{\Gamma}+\frac{\epsilon}{2}}^{\epsilon_{\Gamma}-\frac{\epsilon}{2}}\sfrac{d^j}{dx^j}\psi_{\epsilon/2}(x-u)\,du\right|^2\,dx\leq 2\epsilon_{\Gamma}\int_{-\frac{\epsilon}{2}}^{\frac{\epsilon}{2}} |\psi_{\epsilon/2}^{(j)}(u)|^2\,du,
\end{align*}
so $\|\chi_{\epsilon}\|_{W^2}\ll_{\Gamma} \|\psi_{\epsilon}\|_{W^2}\ll \epsilon^{-\frac{7}{2}}$.
This choice of $\chi_{\epsilon}$ and the fact that $f(\Gamma h)\geq 0$ for all $\Gamma h\in\GaG$ gives
\begin{equation}\label{ineq1}
\int_{-T}^T \!f(\Gamma g n_t)\,dt= \frac{T}{\epsilon_{\Gamma}}\int_{-\epsilon_{\Gamma}}^{\epsilon_{\Gamma}} \!f\big(\Gamma ga_{T/\epsilon_{\Gamma}} n_{ t}a_{T/\epsilon_{\Gamma}}^{-1}\big)\,dt\geq  \frac{T}{\epsilon_{\Gamma}}\int_{-\epsilon_{\Gamma}}^{\epsilon_{\Gamma}} \!f\big(\Gamma ga_{T/\epsilon_{\Gamma}} n_{ t}a_{T/\epsilon_{\Gamma}}^{-1}\big)\chi_{\epsilon}(t)\,dt,
\end{equation}
and
\begin{align}\label{ineq2}
\int_{-T}^T \!f(\Gamma g n_t)\,dt= \frac{T}{\epsilon_{\Gamma}-\epsilon}&\int_{-(\epsilon_{\Gamma}-\epsilon)}^{\epsilon_{\Gamma}-\epsilon} \!f\big(\Gamma ga_{T/(\epsilon_{\Gamma}-\epsilon)} n_{ t}a_{T/(\epsilon_{\Gamma}-\epsilon)}^{-1}\big)\,dt
\\\notag&\leq  \frac{T}{\epsilon_{\Gamma}-\epsilon}\int_{-\epsilon_{\Gamma}}^{\epsilon_{\Gamma}} \!f\big(\Gamma ga_{T/(\epsilon_{\Gamma}-\epsilon)} n_{ t}a_{T/(\epsilon_{\Gamma}-\epsilon)}^{-1}\big)\chi_{\epsilon}(t)\,dt.
\end{align}
Define $\Omega_{\Gamma g}:= \overline{\Omega \cup \lbrace \Gamma g a_y\,\:\, y\geq 1\rbrace}$. Since $\Gamma g \in \GaG_{\mathrm{rad}}$ and $\Gamma $ is convex cocompact, $\Omega_{\Gamma g}$ is compact. Assuming $T\geq \epsilon_{\Gamma}$ then allows us to apply Theorem \ref{translates}: for $r\leq \epsilon_{\Gamma}$, 
\begin{align*}
&\frac{T}{r}\int_{-\epsilon_{\Gamma}}^{\epsilon_{\Gamma}} \!f\big(\Gamma ga_{T/r} n_{ t}a_{T/r}^{-1}\big)\chi_{\epsilon}(t)\,dt
\\&=\frac{T}{r}\left( \sfrac{m_{\Gamma}^{\mathrm{BR}}(f)\mu_{\Gamma ga_{T/r}N}^{\mathrm{PS}}(\chi_{\epsilon})}{m_{\Gamma}^{\mathrm{BMS}}(\GaG)}\left( \sfrac{r}{T}\right)^{1-\delGa} \!\!+\!O_{\Gamma,\Omega_{\Gamma g}}\!\left(\!\left( \sfrac{r}{T}\right)^{1-\delGa +\eta_1}\|f\|_{\scrS^3(\GaG)}\lbrace\|\chi_{\epsilon}\|_{W^3}\!+\!\mu_{\Gamma ga_{T/r}N}^{\mathrm{PS}}(\chi_{\epsilon})\rbrace \!\right)
 \right)
\\&=\sfrac{m_{\Gamma}^{\mathrm{BR}}(f)\mu_{\Gamma ga_{T/r}N}^{\mathrm{PS}}(\chi_{\epsilon})}{m_{\Gamma}^{\mathrm{BMS}}(\GaG)}\left(\frac{T}{r}\right)^{\delGa}+O_{\Gamma,\Omega,\Gamma g}\left(\left(\frac{T}{r}\right)^{\delGa-\eta_1}\|f\|_{\scrS^3(\GaG)}\left\lbrace\epsilon^{-\frac{7}{2}}+\mu_{\Gamma ga_{T/r}N}^{\mathrm{PS}}(\chi_{\epsilon})\right\rbrace\right).
\end{align*}
We now observe that
\begin{align*}
\mu^{\mathrm{PS}}_{\Gamma g a_{T/r} N}(\chi_{\epsilon})=\mu^{\mathrm{PS}}_{\Gamma g a_{T/r} N}(B_{r})+O\left( \mu^{\mathrm{PS}}_{\Gamma g a_{T/r} N}(\lbrace t\in\RR\,:\, \min\lbrace \epsilon_{\Gamma}-\epsilon,r\rbrace \leq |t|\leq \epsilon_{\Gamma}\rbrace)\right)
\end{align*}
and so Lemma \ref{mPSscaling} gives
\begin{equation*}
\mu^{\mathrm{PS}}_{\Gamma g a_{T/r} N}(\chi_{\epsilon})\!=\!\!\left( \frac{T}{r}\right)^{-\delGa}\!\!\!\left(\mPSGgN(B_T)+O\left(\mPSGgN\big( \lbrace t\in\RR\,:\, \frac{\min\lbrace \epsilon_{\Gamma}-\epsilon,r\rbrace}{r}T\leq |t|\leq \frac{\epsilon_{\Gamma}}{r}T\rbrace\big) \right)\right).
\end{equation*}
Since our choices of $r$ are $r=\epsilon_{\Gamma}$ and $r=\epsilon_{\Gamma}-\epsilon$, in both cases we have
\begin{align*}
\mPSGgN\big( \lbrace t\in\RR\,:\, \sfrac{\min\lbrace \epsilon_{\Gamma}-\epsilon,r\rbrace}{r}T\leq |t|\leq \frac{\epsilon_{\Gamma}}{r}T\rbrace\big)\leq \mPSGgN\big( \lbrace t\in\RR\,:\, (1-\sfrac{\epsilon} {\epsilon_{\Gamma}})T\leq |t|\leq (1+\sfrac{\epsilon}{\epsilon_{\Gamma}-\epsilon})T\rbrace\big).
\end{align*}
Assuming $\epsilon\leq \min\lbrace \frac{1}{2},\frac{\epsilon_{\Gamma}}{4}\rbrace$, by the definition of $\mPSGgN$ and Lemma \ref{diffmeas1}, we have
\begin{align*}
&\mPSGgN\big( \lbrace t\in\RR\,:\, (1-\sfrac{\epsilon} {\epsilon_{\Gamma}})T\leq |t|\leq (1+\sfrac{\epsilon}{\epsilon_{\Gamma}-\epsilon})T\rbrace\big) 
\\&\qquad\ll_{\Gamma g} T^{2\delGa}\times \nu_i^{(g^{-1}\Gamma g,i)}\big( \lbrace t\in\RR\,:\, (1-\sfrac{\epsilon} {\epsilon_{\Gamma}/2})T\leq |t|\leq (1+\sfrac{\epsilon}{\epsilon_{\Gamma}/2})T\rbrace\big)
\\&\qquad\qquad\ll_{\Gamma g}T^{2\delGa}\times\epsilon^{\delGa}T^{-\delGa}=\epsilon^{\delGa}T^{\delGa}.
\end{align*}
This gives
\begin{align*}
&\frac{T}{r}\int_{-\epsilon_{\Gamma}}^{\epsilon_{\Gamma}} \!f\big(\Gamma ga_{T/r} n_{ t}a_{T/r}^{-1}\big)\chi_{\epsilon}(t)\,dt
=\sfrac{m_{\Gamma}^{\mathrm{BR}}(f)\mPSGgN(B_T)}{m_{\Gamma}^{\mathrm{BMS}}(\GaG)}
\\&\qquad+O_{\Gamma,\Omega,\Gamma g}\bigg(m_{\Gamma}^{\mathrm{BR}}(f)T^{\delGa}\epsilon^{\delGa}
+\|f\|_{\scrS^3(\GaG)}
\left\lbrace T^{\delGa-\eta_1}\epsilon^{-\frac{7}{2}}+T^{-\eta_1}\mPSGgN(B_T) +T^{\delGa-\eta_1}\epsilon^{\delGa}\right\rbrace\bigg).
\end{align*}
Since $\int_{-T}^T \!f\big(\Gamma g n_{ t})\,dt$ is bounded from above and below by the integrals in the right-hand side of this expression (cf.\ \eqref{ineq1} and \eqref{ineq2}), the same must hold for it. Dividing by $\mPSGgN(B_T)$ and using the bounds $m_{\Gamma,N}^{\mathrm{BR}}(f)\ll_{\Gamma,\Omega}\|f\|_{\scrS^2(\GaG)}
$ and $\mPSGgN(B_T)\gg_{\Gamma,\Gamma g} T^{\delGa}$ then yields
\begin{align*}
\frac{1}{\mPSGgN(B_T)}\!\!\int_{-T}^T \!f\big(\Gamma g n_{ t})\,dt\!=\!\sfrac{m_{\Gamma}^{\mathrm{BR}}(f)}{m_{\Gamma}^{\mathrm{BMS}}(\GaG)}
\!+\!O_{\Gamma,\Omega,\Gamma g}\!\bigg(\!\|f\|_{\scrS^3(\GaG)}\!\left\lbrace\!\epsilon^{\delGa}\!+\!T^{-\eta_1}\epsilon^{-\frac{7}{2}}\!+\!T^{-\eta_1} \!+\!T^{-\eta_1}\epsilon^{\delGa}\!\right\rbrace\!\!\bigg).
\end{align*} 
Choosing $\epsilon=T^{-\frac{\eta_1}{\delGa+7/2}}$ gives
\begin{equation}\label{poseffeq}
\frac{1}{\mPSGgN(B_T)}\!\!\int_{-T}^T \!f\big(\Gamma g n_{ t})\,dt=\frac{m_{\Gamma}^{\mathrm{BR}}(f)}{m_{\Gamma}^{\mathrm{BMS}}(\GaG)}
\!+\!O_{\Gamma,\Omega,\Gamma g}\!\bigg(\!\|f\|_{\scrS^3(\GaG)}T^{-\widetilde{\eta}_{\Gamma}}\bigg),
\end{equation}
where $\widetilde{\eta}_{\Gamma}=\frac{\delGa\,\eta_1}{\delGa+7/2}$. Theorem \ref{mainthm2} is thus proved for non-negative functions.

In order to generalize to all functions in $\scrS^3(\Omega)$, we first that note that if $f\in \scrS^3(\Omega)$, then $\Im(f)$, $\Re(f)\in\scrS^3(\Omega)$, and $\|\Re(f)\|_{\scrS^3(\GaG)}\ll \|f\|_{\scrS^3(\GaG)}$ and $\|\Im(f)\|_{\scrS^3(\GaG)}\ll \|f\|_{\scrS^3(\GaG)}$, so by considering the real and imaginary parts it suffices to to extend \eqref{poseffeq} to $\RR$-valued $f\in\scrS^3(\Omega)$. 

By Lemma \ref{diffbdd}, there exists $C=C(\Gamma)$ such that if $\disG{h}{e}<\epsilon$, then $|f(\Gamma g h)-f(\Gamma g)|\leq C\epsilon\|f\|_{\scrS^3(\GaG)}$. Using this, we assume now that $f$ is $\RR$-valued, and for $\epsilon>0$, define sets $\Omega^+_{\epsilon}(f),\,\Omega_-\epsilon(f)\subset \Omega$ by
\begin{equation*}
\Omega_{\epsilon}^{\pm}(f)=\lbrace \Gamma h \in \Omega\,:\, (\pm 1)f(\Gamma h)>C\epsilon\|f\|_{\scrS^3(\GaG)}\rbrace.
\end{equation*}
We now turn again to \cite[Proposition 2.4.7]{KleinMarg}: for all $0<\epsilon<1$ there exists $\rho_{\epsilon}\in C_c^{\infty}(\scrB_{\epsilon})$ such that
\begin{equation*}
\rho_{\epsilon}(h)\geq 0 \quad\forall h\in G,\qquad\int_{\scrB_{\epsilon}}\rho_{\epsilon}(h)\,d\mu_G(h)=1,\qquad \|\rho_{\epsilon}\|_{\scrS^m(G)}\ll \epsilon^{-(m+3/2)},
\end{equation*}
where $\scrS^m(G)$ denotes the $m$-th order $L^2$-Sobolev norm on $G$ (defined analogously to $\scrS^m(\GaG)$). Define functions $\varphi_{f,\epsilon}^{\pm}$ on $\GaG$ by 
\begin{equation*}
\varphi_{f,\epsilon}^{\pm}(\Gamma h)=\mathbbm{1}_{\Omega_{\epsilon}^{\pm}(f)}\ast \rho_{\epsilon/2}(\Gamma h)=\int_{\scrB_{\epsilon/2}}\mathbbm{1}_{\Omega_{\epsilon}^{\pm}}(\Gamma h h')\rho_{\epsilon/2}(h'^{-1})\,d\mu_G(h').
\end{equation*} 
This definition gives $\mathrm{supp}(\varphi_{f,\epsilon}^{\pm})\subset \Omega_{\epsilon/2}^{\pm}(f)\subset \Omega$ and $(\pm1)f(\Gamma h)>C\epsilon\|f\|_{\scrS^3(\GaG)}\Rightarrow \varphi_{f,\epsilon}^{\pm}(\Gamma h)=1$. Note also that \begin{equation*}
\|\varphi_{f,\epsilon}^{\pm}\|_{\scrS^m(\GaG)}
\ll_{\Gamma,\Omega} \epsilon^{-(m+3/2)}.
\end{equation*}
We now use \eqref{poseffeq}:
\begin{align*}
\frac{1}{\mPSGgN(B_T)}\int_{-T}^{T}f(\Gamma g n_t)\,dt=&\frac{1}{\mPSGgN(B_T)}\int_{-T}^{T}\varphi_{f,\epsilon}^{+}(\Gamma g n_t)f(\Gamma g n_t)\,dt
\\&\quad-\frac{1}{\mPSGgN(B_T)}\int_{-T}^{T}\varphi_{f,\epsilon}^{-}(\Gamma g n_t)|f(\Gamma g n_t)|\,dt
\\&\quad+\frac{1}{\mPSGgN(B_T)}\int_{-T}^{T}\big[\mathbbm{1}_{\Omega}-\varphi_{f,\epsilon}^{+}-\varphi_{f,\epsilon}^{-}\big](\Gamma g n_t)f(\Gamma g n_t)\,dt
\end{align*}
\begin{align*}
=\sfrac{m_{\Gamma}^{\mathrm{BR}}(f)}{m_{\Gamma}^{\mathrm{BMS}}(\GaG)}
\!+\!O_{\Gamma,\Omega,\Gamma g}\!\Bigg(&m_{\Gamma}^{\mathrm{BR}}\big((\mathbbm{1}_{\Omega}-\varphi_{f,\epsilon}^{+}-\varphi_{f,\epsilon}^{-})f\big)\!+\!(\|f\varphi_{f,\epsilon}^{+}\|_{\scrS^3(\GaG)}+(\|f\varphi_{f,\epsilon}^{-}\|_{\scrS^3(\GaG)}))T^{-\widetilde{\eta}_{\Gamma}}
\\&\qquad+\frac{1}{\mPSGgN(B_T)}\int_{-T}^{T}\left|\big[\mathbbm{1}_{\Omega}-\varphi_{f,\epsilon}^{+}-\varphi_{f,\epsilon}^{-}\big](\Gamma g n_t)f(\Gamma g n_t)\right|\,dt\Bigg).
\end{align*}
The terms in the ``$O_{\Gamma,\Omega,\Gamma g}$" are dealt with individually:
\begin{align*}
\left|m_{\Gamma}^{\mathrm{BR}}\big((\mathbbm{1}_{\Omega}-\varphi_{f,\epsilon}^{+}-\varphi_{f,\epsilon}^{-})f\big)\right|\leq& \left\|(\mathbbm{1}_{\Omega}-\varphi_{f,\epsilon}^{+}-\varphi_{f,\epsilon}^{-})f\right\|_{L^{\infty}(\GaG)}m_{\Gamma}^{\mathrm{BR}}(\Omega)\leq C\epsilon \|f\|_{\scrS^3(\GaG)}m_{\Gamma}^{\mathrm{BR}}(\Omega)
\\&\ll_{\Gamma,\Omega} \epsilon \|f\|_{\scrS^3(\GaG)},
\end{align*}
\begin{equation*}
\|f\varphi_{f,\epsilon}^{\pm}\|_{\scrS^3(\GaG)}\ll_{\Gamma} \|f\|_{\scrS^4(\GaG)}\|\varphi_{f,\epsilon}^{\pm}\|_{\scrS^3(\GaG)}\ll_{\Gamma,\Omega} \|f\|_{\scrS^4(\GaG)} \epsilon^{-\frac{9}{2}}.
\end{equation*}
To bound $\frac{1}{\mPSGgN(B_T)}\int_{-T}^{T}\left|\big[\mathbbm{1}_{\Omega}-\varphi_{f,\epsilon}^{+}-\varphi_{f,\epsilon}^{-}\big](\Gamma g n_t)f(\Gamma g n_t)\right|\,dt$, we note that
\begin{align*}
\left|\big[\mathbbm{1}_{\Omega}-\varphi_{f,\epsilon}^{+}-\varphi_{f,\epsilon}^{-}\big](\Gamma g n_t)f(\Gamma g n_t)\right|\leq C\epsilon \|f\|_{\scrS^3(\GaG)}\mathbbm{1}_{\Omega}(\Gamma g n_t)\leq C\epsilon  \|f\|_{\scrS^3(\GaG)} [\mathbbm{1}_{\Omega}\ast\rho_{\epsilon}](\Gamma g n_t).
\end{align*}
Now, $\mathrm{supp}(\mathbbm{1}_{\Omega}\ast\rho_{\epsilon})=\Omega\scrB_{\epsilon}$, $\|\mathbbm{1}_{\Omega}\ast\rho_{\epsilon}\|_{L^{\infty}(\GaG)}=1$, and $\|\mathbbm{1}_{\Omega}\ast\rho_{\epsilon}\|_{\scrS^3(\GaG)}\ll_{\Gamma,\Omega} \epsilon^{-\frac{9}{2}}$. We apply \eqref{poseffeq} again:
\begin{align*}
\frac{1}{\mPSGgN(B_T)}&\int_{-T}^{T}\left|\big[\mathbbm{1}_{\Omega}-\varphi_{f,\epsilon}^{+}-\varphi_{f,\epsilon}^{-}\big](\Gamma g n_t)f(\Gamma g n_t)\right|\,dt
\\& \ll_{\Gamma} \frac{\epsilon \|f\|_{\scrS^3(\GaG)}}{\mPSGgN(B_T)}\int_{-T}^{T}[\mathbbm{1}_{\Omega}\ast\rho_{\epsilon}](\Gamma g n_t)\,dt
\\&=\epsilon\|f\|_{\scrS^3(\GaG)}\left(\frac{m_{\Gamma}^{\mathrm{BR}}(\mathbbm{1}_{\Omega}\ast\rho_{\epsilon})}{m_{\Gamma}^{\mathrm{BMS}}(\GaG)}
\!+\!O_{\Gamma,\Omega\scrB_{\epsilon},\Gamma g}\!\bigg(\!\|\mathbbm{1}_{\Omega}\ast\rho_{\epsilon}\|_{\scrS^3(\GaG)}T^{-\widetilde{\eta}_{\Gamma}}\bigg) \right) 
\\&=O_{\Gamma,\Omega,\Gamma g}\left( \epsilon\|f\|_{\scrS^3(\GaG)}\left\lbrace 1 + \epsilon^{-\frac{9}{2}}T^{-\widetilde{\eta}}
\right\rbrace\right).
\end{align*}
In total, we have
\begin{align*}
\frac{1}{\mPSGgN(B_T)}\int_{-T}^{T}f(\Gamma g n_t)\,dt= \frac{m_{\Gamma}^{\mathrm{BR}}(f)}{m_{\Gamma}^{\mathrm{BMS}}(\GaG)}+O_{\Gamma,\Omega,\Gamma g}\!\bigg(\|f\|_{\scrS^4(\GaG)}\left\lbrace \epsilon + \epsilon^{-\frac{9}{2}}T^{-\widetilde{\eta}}\right\rbrace\bigg).
\end{align*}
Letting $\epsilon=T^{-\frac{2\widetilde{\eta}}{11}}$ completes the proof, with $\eta_{\Gamma}=\frac{2\widetilde{\eta}}{11}=\frac{2\delGa\eta_1}{11(\delGa+7/2)}$.

\hspace{425.5pt}\qedsymbol

\end{document}